\documentclass[11pt]{amsart}

\usepackage[dvipsnames,table]{xcolor}

\usepackage{ytableau,tikz,tikz-cd,varwidth, hyperref, amssymb,mathtools,mathscinet}
\usetikzlibrary{calc,positioning, matrix, shapes,arrows}
\usepackage[enableskew]{youngtab}
\usepackage{geometry}
\usepackage{enumerate}
\usepackage[all,cmtip]{xy}
\usepackage{comment}
\usepackage{fullpage}

\newtheorem{maintheorem}{Theorem}

\newtheorem{thm}{Theorem}[section]

\theoremstyle{definition}

\newcommand{\ZZ}{\mathbb{Z}}
\newcommand{\RR}{\mathbb{R}}

\newcommand{\CC}{\mathbb{C}}
\newcommand{\QQ}{\mathbb{Q}}

\newcommand{\cA}{\mathcal{A}}

\newcommand{\cI}{\mathcal{I}}

\newcommand{\cM}{\mathcal{M}}

\renewcommand{\comment}[1]{}
\newcommand{\on}{\operatorname}
\newcommand{\Gr}{\operatorname{Gr}}

\newcommand{\ov}{\overline}

\newcommand{\Aut}{\on{Aut}}

\newcommand{\GL}{\mathrm{GL}}
\newcommand{\SL}{\mathrm{SL}}

\newcommand{\bu}{\bullet}

\newcommand{\xra}{\xrightarrow}

\newcommand{\wt}{\widetilde}

\newcommand{\Cones}{\on{Cones}}

\newcommand{\Satake}{\mathcal{A}_g^{\mathrm{Sat}}}

\newcommand{\col}{\colon}

\newcommand{\hide}[1]{}

\newcommand{\perf}{\operatorname{P}}

\newcommand{\Span}{\operatorname{span}}
\newcommand{\rank}{\operatorname{rank}}
\newcommand{\Hom}{\operatorname{Hom}}
\newcommand{\PD}{\Omega}
\newcommand{\PDrt}{\Omega^{\mathrm{rt}}}

\newcommand{\Ag}{\mathcal{A}_g}
\newcommand{\Agtrop}[1]{A_g^{\mathrm{trop},#1}}
\newcommand{\AgPer}{\mathcal{A}_g^{\mathrm{perf}}}

\newcommand{\Lk}{L}

\newcommand{\down}[2]{\xymatrix@R=6mm@C=2mm{
#1\ar[d]\\ #2
}}

\newcommand{\downlabel}[3]{\xymatrix@R=6mm@C=2mm{
{#1}\ar[d]^<<<{#3} \\ #2
}}

\newcommand{\squarediagram}[4]{\xymatrix@R=8mm@C=8mm{
#1\ar[d]\ar[r] & #2\ar[d] \\ #3\ar[r] &#4
}}
\newcommand{\squarediagrammapsto}[4]{\xymatrix@R=8mm@C=8mm{
#1\ar@{|->}[d]\ar@{|->}[r] & #2\ar@{|->}[d] \\ #3\ar@{|->}[r] &#4
}}
\newcommand{\squarediagramlabel}[8]{\xymatrix@R=8mm@C=8mm{
#1\ar[d]_{#6}\ar[r]^{#5} & #2\ar[d]^{#7} \\ #3\ar[r]^{#8} &#4
}}

\newcommand{\isocelesdown}[3]{\xymatrix@R=6mm@C=0mm{
& {#1}\ar[dl] \ar[dr] & \\
{#2} \ar[rr] && {#3}
}}

\newcommand{\isocelesdownlabel}[6]{\xymatrix@R=6mm@C=0mm{
& {#1}\ar[dl]_<<<<{#4} \ar[dr]^<<<<{#5} & \\
{#2} \ar[rr]_{#6} && {#3}
}}

\newcommand{\isocelesup}[3]{\xymatrix@R=6mm@C=0mm{
 #1\ar[rr]\ar[dr]  && #2\ar[dl] \\
 & #3 &
}}

\newcommand{\isocelesuplabel}[6]{\xymatrix@R=6mm@C=0mm{
 #1\ar[rr]^{{#4}} \ar[dr]_<<<{#5} && #2\ar[dl]^<<<{#6} \\
 & #3 &
}}

\newtheorem{Definition}[thm]{Definition}
\newenvironment{definition}
  {\begin{Definition}\rm}{\end{Definition}}

\newtheorem{Example}[thm]{Example}
\newenvironment{example}
  {\begin{Example}\rm}{\end{Example}}
  
\newtheorem{Exercise}[thm]{Exercise}
\newenvironment{exercise}
  {\begin{Exercise}\rm}{\end{Exercise}}

\newtheorem{Fact}[thm]{Fact}
\newenvironment{fact}
  {\begin{Fact}\rm}{\end{Fact}}

\newtheorem{Theorem}[thm]{Theorem}
\newenvironment{theorem}
  {\begin{Theorem}\rm}{\end{Theorem}}
  
\newtheorem{Lemma}[thm]{Lemma}
\newenvironment{lemma}
  {\begin{Lemma}\rm}{\end{Lemma}}

\newtheorem{Remark}[thm]{Remark}
\newenvironment{remark}
  {\begin{Remark}\rm}{\end{Remark}}

\newtheorem{Proposition}[thm]{Proposition}
\newenvironment{proposition}
  {\begin{Proposition}\rm}{\end{Proposition}}

\newtheorem{Corollary}[thm]{Corollary}
\newenvironment{corollary}
  {\begin{Corollary}\rm}{\end{Corollary}}

\newtheorem{Question}[thm]{Question}
\newenvironment{question}
  {\begin{Question}\rm}{\end{Question}}

\newtheorem{Conjecture}[thm]{Conjecture}
\newenvironment{conjecture}
  {\begin{Conjecture}\rm}{\end{Conjecture}}

\newtheorem{Construction}[thm]{Construction}
\newenvironment{construction}
  {\begin{Construction}\rm}{\end{Construction}}
  
\theoremstyle{remark}

\newcommand \defnow[1]{\begin{definition}{#1}\end{definition}}

\newcommand \lemnow[1]{\begin{lemma}{#1}\end{lemma}}

\newcommand \thmnow[1]{\begin{theorem}{#1}\end{theorem}}
\newcommand \proofnow[1]{\begin{proof}{#1}\end{proof}}
\newcommand \remnow[1]{\begin{remark}{#1}\end{remark}}
\newcommand \propnow[1]{\begin{proposition}{#1}\end{proposition}}
\newcommand \cornow[1]{\begin{corollary}{#1}\end{corollary}}

\newcommand \enumnow[1]{\begin{enumerate}{#1}\end{enumerate}}

\definecolor{mypink}{RGB}{219, 48, 122}

\title{On the top-weight rational cohomology of $\mathcal{A}_g$}

\author[M. Brandt]{Madeline Brandt}
\address{Department of Mathematics, Brown University, 
Box 1917,  
Providence, RI 02912}
\email{\href{mailto:madeline_brandt@brown.edu}{madeline\_brandt@brown.edu}}

\author[J. Bruce]{Juliette Bruce}
\address{Department of Mathematics, University of California, Berkeley, 970 Evans Hall, Berkeley, CA 94720}
\email{\href{mailto:juliette.bruce@berkeley.edu}{juliette.bruce@berkeley.edu}}

\author[M. Chan]{Melody Chan}\address{Department of Mathematics, Brown University, Box
1917, Providence, RI 02912}\email{\href{mailto:melody_chan@brown.edu}{melody\_chan@brown.edu}}

\author[M. Melo]{Margarida Melo}
\address{Department of Mathematics and Physics, Universit\`a Roma Tre, Largo San Leonardo Murialdo, 00146 Roma, Italy}
\email{\href{mailto:melo@mat.uniroma3.it}{melo@mat.uniroma3.it}}

\author[G. Moreland]{Gwyneth Moreland}
\address{Department of Mathematics, Harvard University, 1 Oxford St, Cambridge, MA 02138}
\email{\href{mailto:gwynm@math.harvard.edu}{gwynm@math.harvard.edu}}

\author[C. Wolfe]{Corey Wolfe}
\address{Department of Mathematics, Tulane University, New Orleans, LA}
\email{\href{mailto:cwolfe@tulane.edu}{cwolfe@tulane.edu}}

\begin{document}
\maketitle
\vspace{-.5in}
\begin{abstract}
    We compute the top-weight rational cohomology of $\cA_g$ for $g=5$, $6$, and $7$,  and we give some vanishing results for the top-weight rational cohomology of $\cA_8, \cA_9,$ and $ \cA_{10}$. 
    When $g=5$ and $g=7$, we exhibit nonzero cohomology groups of $\cA_g$ in odd degree, 
    thus answering a question highlighted by Grushevsky. Our methods develop the relationship between the top-weight cohomology of $\cA_g$ and the homology of the link of the moduli space of principally polarized tropical abelian varieties of rank $g$. To compute the latter we use the Voronoi complexes used by Elbaz-Vincent-Gangl-Soul\'e. 
        In this way, our results
     make a precise connection between the rational cohomology of $\mathrm{Sp}_{2g}(\ZZ)$ and $\GL_g(\ZZ)$. Our computations also 
    give natural candidates for compactly supported cohomology classes of $\cA_g$ in weight $0$ that produce the stable cohomology classes of the Satake compactification of $\cA_g$ in weight $0$, under the Gysin spectral sequence for the latter space.
\end{abstract}
\vspace{-.05in}

\section{Introduction}


Let $\cA_g$ be the moduli stack of principally polarized complex abelian varieties of dimension $g$.
It is well-known that $\cA_g$ is a separated Deligne-Mumford stack, isomorphic to  the quotient 
of the Siegel upper half plane $\mathbb{H}_g$ under the action of the integral symplectic group $\mathrm{Sp}_{2g}(\ZZ)$. Therefore $\mathcal{A}_g$ is smooth of dimension $d=\binom{g+1}{2}$, but it is not proper for $g>0$. Since 
$\cA_g$ is a 
complex 
algebraic variety, the rational cohomology groups of $\cA_g$ admit a weight filtration in the sense of mixed Hodge theory, with graded pieces $\Gr_{j}^W\!H^\bullet(\mathcal{A}_g;\mathbb{Q})$ which may appear for $j$ from
$0$ to $2d$. We refer to the piece of weight $2d$ as the \emph{top-weight rational cohomology} of $\cA_g$.

The orbifold Euler characteristic and the stable cohomology of $\cA_g$ are classically understood \cite{harder-gauss, borel-stable}. 
However, the full cohomology ring $H^{\bullet}(\cA_g;\QQ)$ is a mystery even for small $g$. The cases when $g\le 2$ are classically known, and the case when $g=3$ is the work of Hain \cite{hain-rational}.  The full cohomology ring for $g\ge 4$ is already unknown, though when $g=4$, much information can be determined from \cite{hulek-tommasi-cohomology,hulek-tommasi-topology}, where the complete Betti tables for both the Voronoi and the perfect cone compactifications of $\cA_4$ are computed. In particular, the top-weight cohomology of $\cA_4$ vanishes; see Remark~\ref{rem:a4vanishes}.

We compute the top-weight rational cohomology of $\mathcal{A}_g$ 
for $2\leq g\leq 7$. For $g=3$ and $4$, our computations agree with the above-mentioned results of Hain and Hulek-Tommasi, respectively.  Our first main result is then the following Theorem \ref{thm:main}.
\begin{maintheorem}\label{thm:A}
The top-weight rational cohomology of $\mathcal{A}_g$ for $g=5$, $6$, and $7$, is
\begin{align*}
\Gr_{30}^WH^k(\mathcal{A}_5;\mathbb{Q})&=
\begin{cases}
\mathbb{Q} \text{ if }k=15,20,\\
0\text{ else,}
\end{cases} \\
\Gr_{42}^WH^k(\mathcal{A}_6;\mathbb{Q})&=
\begin{cases}
\mathbb{Q} \text{ if }k=30,\\
0\text{ else,}
\end{cases}\\
\Gr_{56}^WH^k(\mathcal{A}_7;\mathbb{Q})&=
\begin{cases}
\mathbb{Q} \text{ if }k=28,33,37,42\\
0\text{ else.}
\end{cases}
\end{align*}
\label{thm:main}
\end{maintheorem}

\noindent This answers an open question of Grushevsky, who asked whether $\mathcal{A}_g$ ever has nonzero odd cohomology \cite[Open Problem 7]{grushevsky-geometry}.


For broader context, recall that from the description of $\cA_g$ as the quotient $[\mathbb{H}_g/\mathrm{Sp}_{2g}(\ZZ)]$, it is a rational classifying space for the integral symplectic group $\mathrm{Sp}_{2g}(\ZZ)$. Thus, $H^*(\cA_g;\QQ) \cong H^*(\mathrm{Sp}_{2g}(\ZZ);\QQ)$. 
The situation is analogous to that of the moduli space of curves $\cM_g$, which is a rational classifying space for the mapping class group $\mathrm{Mod}_g$ via its action on Teichm\"uller space.  Moreover, in both cases, we find ourselves in the advantageous situation that $\cM_g$ and $\cA_g$  are smooth and separated Deligne Mumford stacks with coarse moduli spaces which are algebraic varieties, permitting Deligne's mixed Hodge theory to be applied to study the rational cohomology of these groups.  
The results of this paper use this algebro-geometric perspective to find new nonzero classes in a canonical quotient of $H^*(\mathrm{Sp}_{2g}(\ZZ);\QQ)$: the {\em top-weight} quotient, in the sense of mixed Hodge theory.
(Recall that in general, the rational cohomology of a 
complex algebraic variety $X$ of dimension $d$ admits a weight filtration with graded pieces $\Gr_{j}^W\!H^k(X;\mathbb{Q})$. 
As $\Gr_{j}^W\!H^k(X;\mathbb{Q})$ vanishes whenever $j>2d$, $\Gr_{2d}^W\!H^k(X;\mathbb{Q})$ is referred to as the {\em top weight} part of $H^k(X;\QQ)$.)

Indeed, in this paper,
we develop methods for studying $\cA_g$ that are analogous to those employed in \cite{cgp-graph-homology} for $\cM_g$. 
The moduli spaces $\mathcal{A}_g$ admit toroidal compactifications $\ov{\cA_g}^\Sigma$, which are proper Deligne-Mumford stacks (see \cite[Theorem 5.7]{faltings-chai-degenerations}). The compactifications $\ov{\cA_g}^\Sigma$ are associated to admissible decompositions $\Sigma$ of  $\PDrt_g$, the rational closure of the cone of positive definite quadratic forms in $g$ variables (see Section \ref{S:admissible}).
The same data is also used to construct the moduli space $A_g^{\mathrm{trop},\Sigma}$ of tropical abelian varieties of dimension $g$ in the category of generalized cone complexes (see Section \ref{subsec:Agtrop}). 

Then for any admissible decomposition $\Sigma$ of $\Omega^{\mathrm{rt}}_g$ and for each $i\ge 0$, and writing $LA_g^{\mathrm{trop},\Sigma}$ for the link of the cone point of $A_g^{\mathrm{trop},\Sigma}$, the following canonical identification holds:



\[\widetilde{H}_{i-1}(\Lk \Agtrop{\Sigma};\QQ) \cong \Gr^W_{2d}H^{2d-i}(\cA_g;\QQ).\]
This statement can be deduced from \cite[Corollary 2.9]{odaka-oshima-collapsing} (see pp.~24--25 of op.~cit.); because the language is different and for self-containedness, we give a short proof in Theorem~\ref{thm:correspondence}. 
Briefly, there exist admissible decompositions $\Sigma$ for which $\ov{\cA}_g^\Sigma $ is a smooth simple normal crossings 
compactification of $\cA_g$ whose boundary complex is identified with $LA_g^{\mathrm{trop},\Sigma}$. 
However, the  homeomorphism type of  $LA_g^{\mathrm{trop},\Sigma}$ is  independent of $\Sigma$: see Section 3 or \cite[A.14]{odaka-tropical}. 
The conclusion follows by applying the generalization to Deligne-Mumford stacks, spelled out in \cite{cgp-graph-homology}, of Deligne's comparison theorems in mixed Hodge theory (see Theorem \ref{thm:correspondence}).

We then compute the topology of $A_g^{\mathrm{trop},\Sigma}$  by considering the \emph{perfect} or \emph{first Voronoi} toroidal compactification $\ov{\mathcal{A}_g}^{\perf}$ and its tropical version $A_g^{\mathrm{trop},\perf}$, associated to the \emph{perfect cone decomposition} (Fact \ref{F:main-Per}).
This decomposition is very well studied and enjoys interesting combinatorial properties, which are well-suited for our computations.
We identify the homology of the link of $A_g^{\mathrm{trop},\perf}$ with the homology of the \emph{perfect chain complex} $P_\bullet^{(g)}$ (Definition \ref{defn:perfect-complex}, Proposition \ref{prop:pg-cellular-chain}), using the framework of cellular chain complexes of symmetric CW-complexes due to Allcock-Corey-Payne \cite{allcock-corey-payne-tropical}.

To compute the homology of the complex $P_\bullet^{(g)}$ we use a related complex $V_\bullet^{(g)}$, called the Voronoi complex. This was introduced in  
 \cite{elbaz-vincent-gangl-soule-perfect, ls78} to compute the cohomology of the modular groups $\GL_{g}(\mathbb Z)$ and $\SL_g(\mathbb Z)$.
They use the perfect form cell decomposition of $\PDrt_g$, which is invariant under the action of each of these groups, and then 
relate the equivariant homology of
$\PDrt_g$ modulo its boundary with the cohomology of $\GL_{g}(\mathbb Z)$ and $\SL_g(\mathbb Z)$, respectively. For this purpose, the homology of  $V_\bullet^{(4)}$ was computed  by Lee and Szczarba in \cite{ls78} for $\SL_4(\ZZ)$ (and we adapt this computation to the case of $\GL_4(\ZZ)$ in this paper), while for $g=5,6, 7 $ the complex $V_\bullet^{(g)}$ was computed in \cite{elbaz-vincent-gangl-soule-perfect} with the help of a computer program  using lists of perfect forms for $g\leq 7$ by Jaquet \cite{jaquet-enumeration}. 
 In Theorem \ref{thm:exact} we show that the complexes $P_\bullet^{(g)}$ and $V_\bullet^{(g)}$ sit in an exact sequence
 \begin{equation}\label{ses-intro}
    \begin{tikzcd}
    0\rar{}&P^{(g-1)}_{\bullet} \rar & P^{(g)}_{\bullet} \rar{\pi} & V^{(g)}_{\bullet} \rar & 0. 
    \end{tikzcd}
\end{equation}
This sequence together with the results in \cite{elbaz-vincent-gangl-soule-perfect} 
are then crucial to get our main result.

In Section \ref{subsec:Ig} we consider a subcomplex of $P_\bullet^{(g)}$ called the inflation complex and prove that it is acyclic. Using this result, we show that
$\Gr_{g^2+g}^WH^i(\mathcal{A}_g;\mathbb{Q})=0$ for $i>g^2$, which  recovers  the vanishing in top weight of the rational cohomology of $\cA_g$ in degree above the virtual cohomological dimension (which for $\cA_g$ is equal to  $g^2$).

For $g=8,9,$ and $10$, full calculations of the top-weight cohomology of $\cA_g$ are beyond the scope of our computations. However, our computations for $g=7$ together with a vanishing result of \cite{sikiri2019voronoi} allow us to deduce, in Section~\ref{subsec:g8}, the vanishing of $\Gr^W_{(g+1)g} H^{\bullet}(\cA_g;\QQ)$ in a range slightly larger than what is implied by virtual cohomological dimension bounds. 

\begin{maintheorem}\label{thm:C}
The top-weight rational cohomology of $\mathcal{A}_{8}$, $\mathcal{A}_{9}$, and $\mathcal{A}_{10}$ vanish in the following ranges:
\begin{align*}
    \Gr_{72}^WH^i(\mathcal{A}_{8};\mathbb{Q})&=0 \quad \text{for $i\geq 60$}, \\
    \Gr_{90}^WH^i(\mathcal{A}_{9};\mathbb{Q})&=0 \quad \text{for $i\geq 79$}, \\
    \Gr_{110}^WH^i(\mathcal{A}_{10};\mathbb{Q})&=0 \quad \text{for $i\geq 99$}.
\end{align*}
\end{maintheorem}

To provide some broader context for our main results on $H^*(\cA_g;\QQ) \cong H^*(\mathrm{Sp}_{2g}(\ZZ);\QQ)$, we now highlight two interesting connections: first, to the stable cohomology of Satake compactifications, and second, to the cohomology of general linear groups $\GL_g(\ZZ)$.  More details appear in Section~\ref{S:relations}.

\medskip

\noindent {\bf Relationship with the stable cohomology of $\Satake$.}  
By Poincar\'e duality,
the top-weight cohomology of $\cA_g$
studied in this paper admits a perfect pairing with weight $0$ compactly supported cohomology of $\cA_g$. 
These weight $0$ classes, in turn, have an interesting, not yet fully understood relationship with the stable cohomology ring of the Satake compactification $\cA_g^{\mathrm{Sat}}$, whose structure was first understood by Charney-Lee \cite{charney-lee-cohomology}.  

Indeed, the stable cohomology ring of $\Satake$ is freelly generated by extensions of the well-known odd $\lambda$-classes and by less understood classes $y_6,\ y_{10},\ y_{14},\ldots$ 
which were proven to be of weight $0$ by Chen-Looijenga in \cite{chen-looijenga-stable}.
This predicts the existence of infinitely many top-weight cohomology classes of $\mathcal{A}_g$ as $g$ grows.
More precisely, the classes found in the present paper, with Poincar\'e duality applied, give natural candidates for the ``sources'' of the $y_j$'s in the sense of persisting in a Gysin spectral sequence relating the compactly supported cohomology groups of the space $\cA_g^\mathrm{Sat}$ and those of the spaces $\cA_{g'}$ for $g'\le g$.  See Table~\ref{tab:E1} at the end of the paper for a summary of everything that is known on the $E_1$ page of this spectral sequence in weight 0.

This connection was explained to us by O.~Tommasi and provides significant additional interest in our main results; we discuss it in detail in Section~\ref{S:relations}.

\medskip

\noindent {\bf Relationship with the rational cohomology of $\GL_g(\ZZ)$.}
Second, we would like to emphasize the connection between $H^*(\mathrm{Sp}_{2g}(\ZZ);\QQ)$ and $H^*(\GL_g(\ZZ);\QQ)$ provided by our main results. The possibility of such a connection is essentially present in \cite{amrt}, but the precise connection employed in this paper, which is a key step in proving our main Theorems~\ref{thm:A} and~\ref{thm:C}, has been underutilized in the literature.

Indeed, Theorem \ref{thm:exact} of this paper shows the exactness of the sequence \eqref{ses-intro} relating the perfect complexes $P^{(g-1)}$ and $P^{(g)}$ on the one hand, and the Voronoi complexes $V^{(g)}$.  Again, these complexes are related to $H^*(\mathrm{Sp}_{2g}(\ZZ);\QQ)$ and $H^*(\GL_g(\ZZ);\QQ)$ respectively: precisely, for all $k$,
\[H^{{g\choose 2} - k}(\GL_g(\ZZ);\QQ) \cong H_{k+g-1}(V^{(g)})\]
and
\[H_{k-1}(P^{(g)}) \cong \Gr^W_{g^2+g} H^{g^2+g-k}(\cA_g;\QQ)\twoheadleftarrow H^{g^2+g-k}(\cA_g;\QQ).\]
(See \cite{soule-3-torsion}, \cite[\S3.4]{elbaz-vincent-gangl-soule-perfect} and Proposition~\ref{prop:pg-cellular-chain}, respectively).
For example, in view of exactness of \eqref{ses-intro}, it is immediately possible to pass vanishing results on the top weight quotient of $H^*(\cA_g;\QQ)$ and vanishing results on $H^*(\GL_g(\ZZ);\QQ)$ back and forth.  For instance, recall that Church-Farb-Putman conjectured \cite[Conjecture 2]{church-farb-putman-stability} that 
\[H^{{g\choose 2} - i}(\SL_g(\ZZ);\QQ) = 0 \text{ for all }i<g-1,\]
which implies the analogous statement for $\GL_g(\ZZ)$.  The conjecture is true for $i=0$ by \cite{lee-szczarba-homology}, for $i=1$ by \cite{church-putman-codimension}, and for $i=2$ by the recent preprint \cite{bruck-miller-patzt-sroka-wilson-codimension}, which appeared after the original version of this paper.  As corollaries of these results and the results of this paper, we thus have, for all $g>0$,
\begin{corollary}
\[\Gr^W_{g^2+g} H^{g^2-k} (\cA_g;\QQ) = 0 \text{ when } k\le 2,\]
\end{corollary}
\noindent which agrees with the $g=9$ and $g=10$ vanishing results in Theorem~\ref{thm:C}.
More generally, the Church-Farb-Putman conjecture would imply that \[\Gr^W_{g^2+g} H^{g^2-i}(\cA_g;\QQ) = 0 \text{ whenever } i<g-1.\]
That is, it would imply vanishing of $E_1^{p,q}$ in the spectral sequence in Table~\ref{tab:E1} for all $q<p-1$; see Section~\ref{S:relations}.

It would be very interesting to find connections to the cohomology of $\GL_g(\ZZ)$ that go deeper in the weight filtration on $H^*(\cA_g;\QQ)$.

\bigskip

The paper is organized as follows. 
In Section \ref{sec:preliminaries}, we give the necessary preliminaries. This includes a discussion of generalized cone complexes, their links, and their homology. We then discuss admissible decompositions of the rational closure of the set of positive definite quadratic forms, and focus in particular on the perfect cone decomposition. We also give a brief introduction to matroids and to perfect cones associated to matroids. Then, we give some background on the tropical moduli space $A_g^{\rm{trop},\Sigma}$, and on the construction of toroidal compactifications of $\Ag$ out of  admissible decompositions.

In Section \ref{S:comparison}, we prove 
Theorem 
\ref{thm:correspondence}), which relates the top-weight cohomology of $\mathcal{A}_g$ to the reduced rational homology of the link of $A_g^{\rm{trop},\Sigma}$.
In Section \ref{sec:chaincomplexes}, we show that the perfect chain complex $P^{(g)}_\bullet$ computes the top-weight cohomology of $\Ag$ (Proposition~\ref{prop:pg-cellular-chain}). We also relate this chain complex to the Voronoi complex $V_\bullet^{(g)}$ (Theorem \ref{thm:exact}).
In Section~\ref{subsec:Ig}, we introduce the inflation subcomplex, which we show is acyclic in Theorem  \ref{thm:Ig-acyclic}.
We prove an analogous result for the coloop subcomplex $C^{(g)}_\bullet$ of the regular matroid complex $R^{(g)}_\bullet$, which may be useful for future results.  

In Section \ref{sec:computations-of-cohomology}, we put together the results obtained in Section \ref{sec:chaincomplexes} with the computations of \cite{ls78} for $g=4$ and \cite{elbaz-vincent-gangl-soule-perfect} in $g=5,6,$ and $7$ to describe the top-weight cohomology of $\Ag$ for $g=4,5,6,$ and $7$ and to give the above mentioned bound for the vanishing of the cohomology of $\mathcal{A}_{g}$ in top weight for $g=8,9,$ and $10$. This proves Theorems~\ref{thm:A} and \ref{thm:C}.
In Section~\ref{S:relations}, we discuss the relationship with the stable cohomology of the Satake compactification, including some open questions which are partially addressed by our main results and which deserve further attention.

\bigskip

\noindent {\bf Acknowledgments.}  We thank ICERM for supporting the Women in Algebraic Geometry Workshop, where this collaboration was initiated. 
The second author is grateful for the support of the Mathematical Sciences Research Institute in Berkeley, California, where she was in residence for the Fall 2020 semester. We are grateful to Philippe Elbaz-Vincent, Herbert Gangl, and Christophe Soul\'{e} for detailed answers over email on several aspects of their work \cite{elbaz-vincent-gangl-soule-perfect} which made this paper possible. We thank S{\o}ren Galatius and Samuel Grushevsky for helpful contextual conversations, and especially Sam Payne for answering several questions and sharing ideas which informed several parts of this paper. Additionally, we thank Daniel Corey, Richard Hain, Klaus Hulek, and Yuji Odaka for providing useful feedback on an early draft of this article, and Francis Brown and Alexander Kupers for additional comments. We are very grateful to Orsola Tommasi for detailed comments on a preliminary version, and in particular sharing her insight on the connection to the stable cohomology ring of Satake compactifications in $\cA_g$, which we have summarized in Section~\ref{S:relations}.  Finally, we thank the anonymous referees for a careful reading and helpful comments.

MB is supported by the National Science Foundation under Award No. 2001739. JB was partially supported by the National Science Foundation under Award Nos. DMS-1502553, DMS-1440140, and NSF MSPRF DMS-2002239. MC is supported by NSF DMS-1701924, NSF CAREER DMS-1844768, and a Sloan Research Fellowship. MM is supported by  MIUR via the Excellence Department Project awarded to the Department of Mathematics and Physics of Roma Tre, by the project  PRIN2017SSNZAW: Advances in Moduli Theory and Birational Classification and is a member of the Centre for Mathematics of the University of
Coimbra -- UIDB/00324/2020, funded by the Portuguese Government through FCT/MCTES. GM is supported by the National Science Foundation under DGE-1745303. Any opinions, findings, and conclusions or recommendations expressed in this material are those of the author(s) and do not necessarily reflect the views of the National Science Foundation. 

\section{Preliminaries}
\label{sec:preliminaries}

In this section we give preliminaries and introduce notation.

\subsection{Cones and generalized cone complexes}\label{subsec:cones}

A \emph{rational polyhedral cone} $\sigma$ in $\RR^g$ (or just a \emph{cone}, for simplicity) is the non-negative real span of a finite set of integer vectors $v_1,v_{2},\ldots, v_n\in\ZZ^g$, 
\[
\sigma\coloneqq\RR_{\geq0}\langle v_{1},v_{2},\ldots,v_{n}\rangle \coloneqq \left\{\sum_{i=1}^n \lambda_i v_i : \lambda_i\in \RR_{\ge 0}\right\}. 
\]
We assume all cones $\sigma \subset \RR^g$ are {\em strongly convex}, meaning $\sigma$ contains no nonzero linear subspaces of $\RR^g$.  The {\em dimension} of $\sigma$ is the dimension of its linear span.  The cone $\sigma$ is said to be \emph{smooth} if it is possible to choose the generating vectors $v_1,\dots, v_n$ so that they are a subset of a $\ZZ$-basis of $\ZZ^g$. Note that some sources refer to what we call smooth cones as \emph{basic} cones. 
A $d$ dimensional cone $\sigma$ is said to be \emph{simplicial} if it is generated by $d$ vectors, which are linearly independent over $\RR$. 
A \emph{face} of $\sigma$ is any nonempty subset of $\sigma$ that minimizes a linear functional on $\RR^g$. Faces of $\sigma$ are themselves rational polyhedral cones. A \emph{facet} is a face of codimension one.

Given cones $\sigma \in\RR^g$ and $\sigma'\in\RR^{g'}$, a \emph{morphism} $\sigma\to\sigma'$ is a continuous map from $\sigma$ to $\sigma'$ obtained as the restriction of a linear map $\RR^g\to\RR^{g'}$ sending $\ZZ^g$ to  $\ZZ^{g'}$. 
A \emph{face morphism} is a morphism of cones $\sigma\to\sigma'$ sending $\sigma$ isomorphically to a face of $\sigma'$. Notice that isomorphisms of cones are examples of face morphisms.
Denote with $\mathrm{Cones}$ the category of cones with face morphisms.

The one-dimensional faces of $\sigma$ are called the {\em extremal rays} of $\sigma$, and there are only finitely many of these.  Given an extremal ray $\rho$ of $\sigma$, the semigroup $\rho \cap \ZZ^g$ is generated by a unique element $u_\rho$ called the {\em ray generator} of $\rho$.
An automorphism of a strongly convex cone permutes its finitely many ray generators, and is uniquely determined by this permutation. So, $\on{Aut}(\sigma)$ is finite.

A \emph{generalized cone complex} (see \cite{acp}) is a topological space with a presentation as a colimit $X\coloneqq\varinjlim_{i\in \mathcal{I}} \sigma_i$ of an arbitrary diagram of cones $\sigma\col \mathcal{I}\to \mathrm{Cones}$, in which all morphisms of cones are face morphisms.  A morphism $(X =\varinjlim_{i\in \mathcal{I}} \sigma_i) \to (X'=\varinjlim_{i\in \mathcal{I}} \sigma'_i)$ is a continuous map $f\colon X\to X'$ such that for each cone $\sigma_i$ in the presentation of $X$, there exists a cone $\sigma'_j$ in the presentation of $X'$ and a morphism of cones $f_i\col \sigma_i\to \sigma'_j$ such that the following diagram commutes.
$$\squarediagramlabel{\sigma_i}{\sigma'_j}{X}{X'}{f_i}{}{}{f}$$
We remark that the category of generalized cone complexes is equivalent to the one of stacky fans as defined in \cite[Def. 2.1.7]{cmv}.

\subsection{Links of generalized cone complexes}\label{subsec:links}

For any cone $\sigma \subset \RR^g$, let us define the {\em link} of $\sigma$ at the origin to be the topological space $\Lk \sigma = (\sigma-\{0\})/\RR_{>0}$, where the action of $\mathbb{R}_{>0}$ is by scalar multiplication. Thus $\Lk \sigma$ is homeomorphic to a closed ball of dimension $\dim \sigma-1$.  A face morphism of cones $\sigma\to\sigma'$ induces a morphism of links $\Lk \sigma\to \Lk \sigma'$, making $\Lk$ a functor from $\mathrm{Cones}$ to topological spaces.

Let $X = \varinjlim_{i\in \cI} \sigma_i$ be a generalized cone complex, where $\sigma\col \cI\to \rm{Cones}$ is a diagram of cones. We define the {\em link} of $X$ as the colimit $$\Lk X = \underrightarrow{\lim} (\Lk\circ \sigma).$$  Thus $\Lk X$ is a topological space, equipped with a colimit presentation as above. In fact, $\Lk X$ is a {\em symmetric CW-complex}, by \cite[Example 2.4]{allcock-corey-payne-tropical}. The definition of symmetric CW-complex generalizes the {\em symmetric $\Delta$-complexes} of \cite{cgp-graph-homology}.  Roughly, a symmetric CW-complex is like a CW-complex, except with closed $n$-balls replaced by quotients thereof by finite subgroups of the orthogonal group $O(n)$.

Let $X$ be a finite generalized cone complex, meaning that the indexing category $\mathcal{I}$ is equivalent to one with a finite number of objects and morphisms.  We now write down a chain complex isomorphic to the {\em cellular chain complex} of $\Lk X$, in the sense of \cite[\S4]{allcock-corey-payne-tropical}, \cite[\S3]{cgp-graph-homology}), whose homology is identified with the singular homology of $\Lk X$.

For each $p\ge -1$, let $\Cones_p(X)$ denote the finite groupoid whose objects are all $(p+1)$-dimensional faces of $\sigma_i$, for all $i\in \cI$, with a morphism $\tau \to \tau'$ for each isomorphism of cones $\phi\col \tau\xra{\cong}\tau'$ such that the following diagram commutes.
$$\isocelesuplabel{\tau}{\tau'}{X}{\phi}{}{}$$

 Let $\tau$ be a cone in $\Cones_p(X)$.  We make use of three compatible notions of orientation found in the literature: (i) an orientation of $\tau$ is an orientation of $L\tau$ \cite{ls78}, (ii) an orientation of $\tau$ is an orientation of the suspension of $L\tau$ \cite{allcock-corey-payne-tropical}, and (iii) an orientation of $\tau$ is an orientation of $\RR\tau$, the $\RR$-linear span of $\tau$; 
i.e., it is a choice of ordered basis for $\RR\tau$, up to a change of basis with positive determinant \cite{elbaz-vincent-gangl-soule-perfect}. For the first two definitions, it is clear that an orientation on $\tau$ induces an orientation on the faces of $\tau$ as well. For the third definition, given a facet $\tau'$ of $\tau$, the induced orientation on $\tau'$ is any one such that the quantity $\epsilon(\tau',\tau)$, defined as follows, is $1$.
 

  Let $B = (v_1,\dots,v_n,v)$ where $B' = (v_1,\dots, v_n)$ is an orientation of $\tau'$ and $v$ is a ray generator of a ray of $\tau$ not contained in $\tau'$. Set $\epsilon(\tau',\tau)$ to be the sign of the orientation of $B$ in the oriented vector space $\mathbb{R}\tau$. Note that this sign does not depend on the choice of $v$. These definitions are compatible, in that a choice of orientation under one definition yields a choice of orientation under the other two, and under this correspondence a cone morphism $\tau \to \sigma$ is orientation-preserving under one definition if it is orientation-preserving under all three.
Say that $\tau\in \Cones_p(X)$ is {\em alternating} if all automorphisms  $\tau\to \tau$ in $\Cones_p(X)$ are orientation-preserving on $\tau$.

Choose a set $\Gamma_p$ of representatives of isomorphism classes of alternating cones in $\Cones_p(X)$, and for each $\tau\in \Gamma_p$ fix an orientation $\omega_{\tau}$ of $\tau$.  If $\rho'$ is a facet of $\tau$, then $\omega_\tau$ induces an orientation of $\rho'$, which we denote $\omega_\tau|_{\rho'}.$
Let $C_p(\Lk X)$ be the $\QQ$-vector space with basis $\Gamma_p$.  We define a differential
$$\partial\col C_p(\Lk X) \to C_{p-1} (\Lk X)$$
by extending linearly on $C_p(\Lk X)$ the following definition: given $\tau\in \Gamma_p$ and $\rho\in \Gamma_{p-1},$ set
$$\partial(\tau)_\rho = \sum_{\rho'} \eta(\rho', \rho)$$
where $\rho'$ ranges over the facets of $\tau$ that are isomorphic in $\Cones_{p-1}(X)$ to $\rho$, and where $\eta(\rho',\rho) = \pm 1$ according to whether an isomorphism $\phi\col \rho'\to \rho$ in $\Cones_p(X)$ takes the orientation $\omega_\tau|_{\rho'}$ to $\omega_{\rho}$ or $-\omega_{\rho}$.  Note that $\eta(\rho',\rho)$ is well defined, i.e., independent of choice of $\phi$, precisely because $\rho$ is alternating.

Let $C_\bu(\Lk X)$ denote the complex
$$\cdots \xra{\partial} C_p(\Lk X) \xra{\partial} C_{p-1}(\Lk X) \xra{\partial} \cdots C_{-1}(\Lk X) \to 0.$$

The main proposition in this subsection is the following.\smallskip

\propnow{\label{prop:cellular-complex}  Let $X$ be a finite generalized cone complex.  We have that $C_\bu(\Lk X)$ is a complex, i.e., $\partial^2 = 0$.  \begin{enumerate}\item If $X$ is connected, we have, for each $p\ge 0$,
$$H_p(C_\bu(\Lk X)) \cong \widetilde{H}_p(\Lk X;\QQ).$$ \item More generally, for each $p>0$, we have canonical isomorphisms
$$H_p(C_\bu(\Lk X)) \cong H_p(\Lk X;\QQ),$$
and for $p=0$ we have $$H_0(C_\bu(\Lk X)) \cong \ker(H_0(\Lk X;\QQ)\to \QQ\Gamma_{-1}).$$\end{enumerate}
}
Proposition~\ref{prop:cellular-complex} follows from \cite[Theorem 4.2]{allcock-corey-payne-tropical}, by tracing through their definition of the cellular chain complex of $\Lk X$.  We give a self-contained proof sketch below.

\begin{proof}[Proof sketch]
Write $\Lk X^{(p)}$ for the $p$-skeleton of $\Lk X$, i.e., the union of the images of $\Lk \sigma$ in $X$, for $\sigma$ ranging over cones of dimension at most $p\!+\!1$ in $X$.
By 
a standard argument analogous to \cite[Theorem 2.2.27]{hatcher},  the complex 
\begin{equation}\label{eq:cellular} \cdots H_{p}(\Lk X^{(p)},\Lk X^{(p-1)};\QQ)\xra{\delta_p} H_{p-1}(\Lk X^{(p-1)},\Lk X^{(p-2)};\QQ) \cdots
\end{equation}
has homology canonically identified with the singular homology of $\Lk X$.
Moreover, 
$$H_{p}(\Lk X^{(p)},\Lk X^{(p-1)};\QQ) \cong \bigoplus_{\tau} H_{p}((\Lk \tau)/\Aut(\tau), (\partial \Lk \tau)/\Aut(\tau);\QQ),$$
where $\tau$ ranges over a set of representatives of isomorphism classes in $\Cones_p(X).$  Here $\Aut(\tau) = \on{Iso}_{\Cones_p(X)}(\tau,\tau)$ is the automorphism group of $\tau$ in $\Cones_p(X)$.  Since $|\!\Aut(\tau)|$ is invertible in $\QQ$, it follows that 
\begin{eqnarray*}H_{p}(\Lk X^{(p)},\Lk X^{(p-1)};\QQ) &\cong& \bigoplus_{\tau} H_{p}(\Lk \tau, \partial \Lk \tau;\QQ)_{\Aut(\tau)},
\end{eqnarray*}
and for each $\tau \in \Cones_p(X)$, we have
$$H_p(\Lk \tau, \partial \Lk \tau;\QQ)_{\Aut(\tau)} \cong \begin{cases}\QQ &\text{if $\tau$ is alternating,}\\0& \text{else.}\end{cases},$$
which identifies $H_{p}(\Lk X^{(p)},\Lk X^{(p-1)};\QQ)$ with $C_p(LX).$
\end{proof}

\remnow{A statement analogous to Proposition~\ref{prop:cellular-complex} holds with $\QQ$ replaced by a commutative ring $R$, if the order of $\Aut(\tau)$ is invertible in $R$ for each $\tau$.}
\remnow{See also the proofs in~\cite[\S3]{ls78}, as well as~\cite[\S3.3]{elbaz-vincent-gangl-soule-perfect}, which are written in the special cases of the {\em Voronoi complexes} for $\SL_g(\ZZ)$ and $\GL_g(\ZZ)$, but apply essentially verbatim to prove Proposition~\ref{prop:cellular-complex}.  The Voronoi complex of $\GL_g(\ZZ)$ plays an important role in this paper.
}

\subsection{Admissible decompositions}\label{S:admissible}

We now introduce admissible decompositions of the rational closure of the set of positive definite quadratic forms, which are used in the construction of toroidal compactifications of the moduli space of abelian varieties, as well as in the construction of the moduli space of tropical abelian varieties. 

We denote by $\mathbb R^{\binom{g+1}{2}}$ the vector space of quadratic forms
in $\mathbb R^g$, which we identify with $g\times g$ symmetric matrices with
coefficients in $\mathbb R$.  We denote by $\PD_g$ the cone in $\mathbb R^{\binom{g+1}{2}}$ of positive
definite quadratic forms. 
We define the rational closure of $\PD_{g}$ to be the set $\PDrt_{g}$
of positive semidefinite quadratic forms whose kernel is defined over $\mathbb{Q}$. 
The group $\GL_g(\ZZ)$ acts on the vector space $\RR^{\binom{g+1}{2}}$ of quadratic forms
by $h\cdot Q\coloneqq h Q h^t$, where $h\in \GL_g(\ZZ)$ and $h^t$ is its transpose.
The cones $\PD_g$ and $\PDrt_g$ are preserved by this action of $\GL_g(\ZZ)$.

\begin{remark}\label{rat-qua}
A positive semidefinite
quadratic form $Q$ in $\RR^g$ belongs to $\PDrt_g$ if and only if
there exists $h\in \GL_g(\ZZ)$ such that
$$hQh^t=
\left(\begin{array}{cc}
Q' & 0 \\
0 &  0 \\
\end{array}\right)
$$
for some positive definite quadratic form $Q'$ in $\RR^{g'}$, with $0\leq g'\leq g$ (see \cite[Sec. 8]{namikawa-toroidal-book}).
\end{remark}

The cones $\PD_g$ and $\PDrt_g$ are not polyhedral cones. However, one can consider decompositions of these spaces into rational polyhedral cones, as in the following definition.

\defnow{\label{decompo}[\cite[Lemma 8.3]{namikawa-toroidal-book}, \cite[Chap. IV.2]{faltings-chai-degenerations}]
An \emph{admissible decomposition} of $\PDrt_g$ is a collection $\Sigma=\{\sigma_{\mu}\}$ of rational polyhedral cones of
$\PDrt_g$ such that:
\begin{enumerate}[(i)]
\item if $\sigma$ is a face of $\sigma_{\mu}\in \Sigma$ then $\sigma\in \Sigma$,
\item the intersection of two cones $\sigma_{\mu}$ and $\sigma_{\nu}$ of $\Sigma$ is a face of both cones,
\item if $\sigma_{\mu}\in \Sigma$ and $h\in \GL_g(\ZZ)$ then $h\sigma_{\mu}h^t\in \Sigma$,
\item 
the set of $\GL_g(\ZZ)$-orbits of cones is finite, and
\item $\cup_{\sigma_{\mu}\in \Sigma} \sigma_{\mu}=\PDrt_g$.
\end{enumerate}
We say that two cones $\sigma_{\mu}, \sigma_{\nu}\in \Sigma$ are equivalent if they are in the same $\GL_g(\ZZ)$-orbit. 
}

There are three known families of admissible decompositions of $\PDrt_g$ described for all $g$: the perfect cone decomposition, the second Voronoi decomposition, and the central cone decomposition (see \cite[Chap. 8]{namikawa-toroidal-book} and the references there).
In this paper, we work with the perfect cone decomposition, which we now describe.

\subsection{The perfect cone decomposition}

Given a positive definite quadratic form $Q$, consider the set of nonzero integral vectors where $Q$ attains its minimum, 
$$M(Q)\coloneqq\{\xi\in \ZZ^g\setminus\{0\}\: : \: Q(\xi)\leq Q(\zeta), \forall\zeta\in \ZZ^g\setminus\{0\} \}. $$
The elements of $M(Q)$ are called the {\em minimal vectors} of $Q$. 
Let $\sigma[Q]$ denote the rational polyhedral subcone of $\PDrt_g$ given by the non-negative linear span of the rank one forms  $\xi\cdot \xi^t\in \PDrt_g$
for elements $\xi$ of $M(Q)$, i.e.
\begin{equation}\label{D:perf-cones}
\sigma[Q]\coloneqq\RR_{\geq 0}\langle \xi\cdot \xi^t\rangle_{\xi\in M(Q)}.
\end{equation}
The \emph{rank} of the cone $\sigma[Q]$ is defined to be the maximum rank of an element of $\sigma[Q]$; 
in fact the rank of $\sigma[Q]$ is exactly the dimension of the span of $M(Q)$, see Lemma~\ref{lem:when-boundary}. 

\begin{fact}\cite{voronoi-nouvelles}
\label{F:main-Per}
The set of cones
$$\Sigma^{\perf}_g\coloneqq \{ \sigma[Q] \: :\:  Q \text{ a positive definite form on }\RR^g \} $$
is an admissible decomposition of $\PDrt_g$, known as the  \emph{perfect cone decomposition}.
\end{fact}

\noindent The quadratic forms $Q$ such that $\sigma[Q]$ has maximal dimension $\binom{g+1}{2}$ are called \emph{perfect},
hence the name of this admissible decomposition.

\begin{example}\label{E:Per-g2}
 Let us compute $\Sigma^{\perf}_2$.  
 In this case, there is a unique perfect form up to $\GL_2(\ZZ)$-equivalence, namely
 \[Q = \left(\begin{matrix}1&1/2\\1/2&1\end{matrix}\right).\]
 One can compute that
 $M(Q)= \{(\pm 1,0), (0,\pm1), (\pm 1, \mp 1)\}.$
 Thus, up to $\GL_2(\ZZ)$-equivalence, there is a unique perfect cone $\sigma[Q]$ of maximal dimension $3$, with ray generators
 $$ \left(\begin{matrix}1&0\\0&0\end{matrix}\right), \quad \left(\begin{matrix}0&0\\0&1\end{matrix}\right),\quad \left(\begin{matrix}1&-1\\-1&1\end{matrix}\right).$$
One may check that for $i \in \{0,1,2\}$, all $i$-dimensional faces of $\sigma[Q]$ are $\GL_2(\ZZ)$-equivalent; hence there is a unique perfect cone of each dimension up to the action of $\GL_2(\ZZ)$. 
 
\end{example}

\begin{remark}\label{R:Per-simpli}
The cones $\sigma[Q] \in \Sigma^{\perf}_g$ need not be simplicial for $g\geq 4$ (see \cite[p. 93]{namikawa-toroidal-book}).
\end{remark}

\subsection{Matroidal Perfect Cones}\label{subsec:matroids}

We now give a brief introduction to matroids and their associated orbits of perfect cones. Further background on matroids can be found in \cite{oxley-matroid}.

\defnow{
A \emph{matroid} $M = (E,\mathcal{C})$ on a finite set $E$ is a 
subset 
$\mathcal{C} \subset \mathcal{P}(E)\setminus\{\emptyset\}$, called the set of \emph{circuits} of $M$, satisfying the following axioms:
\begin{enumerate}
    \item[(C1)] No proper subset of a circuit is a circuit.
    \item[(C2)] If $C_1,C_2\in \mathcal{C}$ are distinct and $c \in C_1 \cap C_2$ then $(C_1 \cup C_2) - \{c\}$ contains a circuit.
\end{enumerate}
}

A matroid $M = (E, \mathcal{C})$ is said to be \emph{simple} if it has no circuits of length 1 or 2. A matroid $M = (E, \mathcal{C})$ is called \emph{representable} over a field $\mathbb{F}$ if there is a matrix $A$ over $\mathbb{F}$ such that $E$ bijects to the columns of $A$ with the circuits $\mathcal{C}$ of $M$ indexing the minimal linearly dependent sets of columns of $A$. 
The matrix $A$ is known as an $\mathbb{F}$-representation of $M$. An \emph{automorphism} of a matroid is a bijection $\phi: E \rightarrow E$ such that for any subset $C \subset E$, $C$ is a circuit of $M$ if and only of $\phi(C)$ is a circuit of $M$.
 
\begin{definition}
A matroid is \emph{regular} if and only if it is representable over every field.
\end{definition}

A matroid $M$ being regular is equivalent to $M$ being representable over $\RR$ by a totally unimodular matrix (i.e., a matrix such that every minor is either $-1, 0,$ or $1$). The \emph{rank} of a regular matroid $M$ is the smallest number $r$ such that $M$ is representable over $\RR$ by a $r\times n$ totally unimodular matrix for some $n$ \cite[Lemma~2.2.21, page 85]{oxley-matroid}.

 \begin{definition}\label{defn:graphic-matroid}
 Let $G$ be a graph. The \emph{graphic matroid} $M(G)$ is the matroid with ground set $E(G)$ whose circuits are subsets of $E(G)$ forming a simple cycle of $G$.
 \end{definition}
 
Since graphic matroids are regular, they are representable over fields of any characteristic. This can be seen directly by constructing the following matrix representing $M(G)$. Fix an orientation of the edges of $G$. Let $A(G)$ be the $|V(G)|\times|E(G)|$ matrix with entries 
 \[A(G)_{ij} =  \begin{cases} 
      0 \text{ if }  v_i \not \in e_j,\\
      -1 \text{ if } v_i \text { is the head of } e_j ,\\
      1 \text{ if } v_i \text{ is the tail of }e_j. 
   \end{cases}
 \]
 This matrix represents the matroid $M(G)$ over any field.
 
\begin{construction}\label{consrt:matroid-cones}
Given a simple, regular matroid $M$ of rank $\leq g$ choose a $g\times n$ totally unimodular matrix $A$ that represents $M$ over $\RR$. Denoting the columns of $A$ by $v_{1},v_{2},\ldots,v_{n}$ we let $\sigma_{A}(M) \subset \PDrt_{g}$ be the rational polyhedral cone:
\[
\sigma_{A}(M)\coloneqq\RR_{\geq0}\left \langle v_{1}v_{1}^{t},v_{2}v_{2}^{t},\ldots,v_{n}v_{n}^{t} \right\rangle.
\]
By \cite[Theorem~4.2.1]{melo-viviani-comparing},
the cone $\sigma_A(M)$ is a perfect cone in $\Sigma_{g}^{\perf}$.
\end{construction}

The cone $\sigma_A(M)$ is uniquely determined by $M$ up to the action of $\GL_{g}(\ZZ)$. In particular, if $A$ and $A'$ are  two different totally unimodular matrices representing $M$ over $\RR$ then there exists an element $h\in\GL_{g}(\ZZ)$ such that $h\sigma_{A}(M) h^{t} = \sigma_{A'}(M)$ (see \cite[Lemma 4.0.5(ii)]{melo-viviani-comparing}). 
We therefore denote the $\GL_{g}(\ZZ)$-orbit of $\sigma_{A}(M)$ by $\sigma(M)$.

In the case of graphic matroids, Construction~\ref{consrt:matroid-cones} can be made very explicit. As this is useful in Section~\ref{sec:computations-of-cohomology}, we take the time to explain it here. Fix $g >0$. We now construct cones of $\Sigma_g^{\perf}$ from graphs on $g+1$ vertices. The rows of the $(g+1)\times|E(G)|$ matrix $A(G)$ as constructed above are linearly dependent. Let $A^*(G)$ be the matrix obtained from $A(G)$ by deleting the last row. The matrices $A(G)$ and $A^*(G)$ are both representations of $M(G)$. 
Let $v_1, \ldots, v_d$ be the columns of $A^*(G)$. Then $\sigma(M(G)) : = \mathbb{R}_{\geq 0}\langle v_1v_1^t, \ldots, v_dv_d^t \rangle \in \Sigma_g^{\perf}$ is a perfect cone (see \cite[Theorem 4.2.1]{melo-viviani-comparing}).

\begin{definition}{}
The \emph{principal cone} is $\sigma_g^{\mathrm{prin}} \coloneqq \sigma(M(K_{g+1}))$, the cone corresponding to the complete graph $K_{g+1}$. 
\end{definition}
When $g=2$, this is the cone discussed in Example \ref{E:Per-g2}. More generally, for arbitrary $g$ the principal cone can be defined as the cone corresponding to the quadratic form
$$
\begin{bmatrix}
1  & 1/2&  \cdots  & 1/2 \\
1/2 & 1 & \cdots &  1/2 \\
\vdots & \vdots &\ddots & \vdots \\
1/2 &1/2 & \cdots & 1
\end{bmatrix}.
$$
These two definitions agree by {\cite[Lemma 6.1.3]{bmv}}.

The faces of $\sigma_g^{\mathrm{prin}}$ may be understood as follows. Since $M(K_{g+1})$ is a simple matroid, the principal cone in $\Sigma_g^{\perf}$ is simplicial by \cite[Theorem 4.4.4(iii)]{bmv}. Therefore, a codimension $i$ face of the principal cone comes from a graph obtained by deleting $i$ edges from $K_{g+1}$.

\begin{remark}
\label{rem:autos}
Automorphisms of the graph $G$ give automorphisms of the matroid $M(G)$, but not all automorphisms of $M(G)$ arise in this way. However, if $G$ is 3-connected, then $\Aut(G) = \Aut(M(G))$ (this is proved by Whitney in \cite{whitney}, see \cite[Lemmas 1 and 2]{hpw72}). The group $\Aut(M(G))$ is isomorphic to the group of permutations of the rays of $\sigma(M(G))$ induced by elements of $GL_g(\mathbb{Z})$ stabilizing $\sigma(M(G))$ \cite[Theorem 5.10]{torelli}.
\end{remark}

\subsection{The tropical moduli space \texorpdfstring{$A_g^\mathrm{trop}$}{Agtrop}}\label{subsec:Agtrop}\label{sec_Agtrop}

We now introduce the moduli space of tropical abelian varieties, which is a generalized cone complex  constructed in \cite{bmv} and later worked out in \cite{cmv}. Our aim is to  compute the homology of the link of $A_g^\mathrm{trop}$, as this is canonically isomorphic to the top-weight rational cohomology of $\cA_{g}$ (see Theorem~\ref{thm:correspondence}).

\defnow{
\mbox{}
A {\em principally polarized tropical abelian variety} (or, for simplicity, just \emph{tropical abelian variety}) of dimension $g$ is a pair $A=(\RR^g/\ZZ^g, Q)$, where $Q$ is a positive semidefinite symmetric bilinear form on $\RR^g$ with rational null space. 
We say that $A=(\RR^g/ \ZZ^g, Q)$ is {\em pure} if $Q$ is positive definite. 
}
Two tropical abelian varieties $(\RR^g/\ZZ^g,Q)$ and $(\RR^g/\ZZ^g,Q')$ are \emph{isomorphic} if there is $h\in\GL_g(\ZZ)$ such that $Q'=hQh^t$.
The set of isomorphism classes of tropical abelian varieties of dimension $g$ is in bijective correspondence with the orbits in $\PDrt_g/ \GL_g(\ZZ)$.

Given an admissible decomposition $\Sigma$ of $\PDrt_g$, one can define a generalized cone complex $A_g^{\mathrm{trop},\Sigma}$ by considering the stratified quotient of $\PDrt_g$ with respect to $\Sigma$ (see  \cite[Def. 2.2.2]{cmv}). 
Precisely, $A_g^{\mathrm{trop},\Sigma}$ is the generalized cone complex obtained as the colimit
$$A_g^{\mathrm{trop},\Sigma}\coloneqq\underrightarrow{\lim}\{\sigma\}_{\sigma\in \Sigma} $$
with arrows given by inclusion of faces composed with the action of the group $\GL_g(\ZZ)$ on $\PDrt_g$: given two cones $\sigma_i$ and $\sigma_j\in \Sigma$ and $h\in\GL_g(\ZZ)$ with $h\sigma_ih^t$ a face of $\sigma_j$, we consider its associated lattice-preserving linear map $L_{i,j,g}:\sigma_i\hookrightarrow \sigma_j$ in the diagram.  The space $A_g^{\mathrm{trop},\Sigma}$  is  the {\em moduli space of tropical abelian varieties} of dimension $g$ with respect to $\Sigma$.

\subsection{Toroidal compactifications of the moduli space \texorpdfstring{$\cA_g$}{Ag}}

In this paper, $\cA_g$ denotes the moduli stack of  principally polarized abelian varieties of dimension $g$. 
It is a smooth Deligne-Mumford algebraic stack of dimension  $d=\binom{g+1}{2}$, and the coarse moduli space of principally polarized  abelian varieties, denoted $A_g$, is a quasiprojective variety.    

The moduli stack $\cA_g$ is not proper for $g>0$, and there are different constructions of compactifications of $\cA_g$. In particular, it is possible to construct normal crossings compactifications of $\cA_g$ via the theory of toroidal compactifications. Both the constructions of $\cA_g$  and of its toroidal compactifications as algebraic stacks were achieved in \cite{amrt} over the complex numbers and in \cite{faltings-chai-degenerations} over an arbitrary base. Even though we work over the complex numbers, we often refer to the constructions in \cite{faltings-chai-degenerations} as these are more conveniently stated within the algebraic category and specifically for moduli of abelian varieties (rather than quotients of bounded symmetric domains as in \cite{amrt}). 

Let $\Sigma$ be an admissible decomposition  of $\PDrt_g$ (in the sense of Definition \ref{decompo}). Then  one may associate to $\Sigma$ a {\em toroidal compactification} $\ov{\Ag}^{\Sigma}$ of  $\cA_g$, which is a proper Deligne-Mumford stack, although in general it is not smooth. The fact that  $\cA_g\subset\ov{\Ag}^{\Sigma}$ is toroidal means that $(\Ag,\ov{\Ag}^{\Sigma})$ is \'etale-locally isomorphic to a torus inside a toric variety.

By construction, the toroidal compactification $\ov{\Ag}^{\Sigma}$ comes with a stratification into locally closed
subsets. These are in order-reversing bijection, with respect to the order relation given by the closure, with the $\GL_g(\ZZ)$-equivalence classes of the relative interiors of the cones in $\Sigma$. For example, the origin of $\PDrt_g$, which is the unique zero-dimensional cone in every admissible decomposition $\Sigma$, corresponds to the open substack $\Ag$, which is the unique stratum of $\ov{\Ag}^{\Sigma}$ of maximal dimension $d$. At the other extreme, the maximal dimensional cones in $\Sigma$ correspond to the zero-dimensional strata of $\ov{\Ag}^{\Sigma}$.

We study the {\em perfect} toroidal compactification $\AgPer = \ov{\Ag}^{\Sigma^{\perf}_g}$ of $\Ag$, i.e., the toroidal compactification of $\Ag$ associated to the perfect cone decomposition $\Sigma^{\perf}_g$. The geometric significance of the perfect cone compactification was highlighted in work of Shepherd-Barron \cite{shepherd-barron-perfect}, who shows that $\AgPer$ is the canonical model of $\Ag$ for $g\geq 12$. 
For our purposes it is particularly nice because the number of strata of codimension $l$ in the boundary of $\AgPer\setminus \cA_g$ is independent of $g$ if $l \leq g$ (see \cite[Prop. 7.1]{grushevsky-hulek-tommasi-stable}).

\section{A comparison theorem for 
$\cA_g$ and $A_g^{\mathrm{trop}}$}\label{S:comparison}

Let $\Sigma$ be any admissible decomposition of $\PDrt_g$. As mentioned above, we can use $\Sigma$ to construct both a toroidal compactification of $\cA_g$, and the generalized cone complex $\Agtrop{\Sigma}$, the moduli space of tropical abelian varieties associated to $\Sigma$. 
In this section, we record the relationship between the homology of $\Agtrop{\Sigma}$ with the top-weight cohomology of $\cA_g$, as deduced from Deligne's comparison theorems and the framework in \cite{cgp-graph-homology}.  This precise relationship was already remarked by Odaka-Oshima \cite[Corollary 2.9]{odaka-oshima-collapsing}, as we explain further in Remark~\ref{rem:msbj}, but it is useful to have a self-contained proof, below. 

\thmnow{
\label{thm:correspondence}
For each $i\ge 0$ and admissible decomposition $\Sigma$, we have a canonical isomorphism
\[\widetilde{H}_{i-1}(\Lk \Agtrop{\Sigma};\QQ) \cong \Gr^W_{2d}H^{2d-i}(\cA_g;\QQ),\]
where $d=\binom{g+1}{2}$ is the complex dimension of $\cA_g$.
}

\begin{proof}
First, by replacing $\Sigma$ with another admissible decomposition of $\PDrt_g$ that refines it, we  may assume that every cone of $\Sigma$ is smooth and that it enjoys the following additional property: for any $h\in \GL_g(\ZZ)$ and $\sigma\in\Sigma$, we have that
$h$ fixes, pointwise, the cone $h\sigma h^t \cap \sigma$. Such a refinement is well known to exist \cite[IV.2, p.~98]{faltings-chai-degenerations}.  For example, one may be obtained by taking the barycentric refinement, which is simplicial, and then taking an appropriate {\em smooth} refinement which can be constructed as in \cite[Theorem 11.1.9]{cox-little-schenck-toric}.  The homeomorphism type of $\Lk \Agtrop{\Sigma}$ is unchanged when passing to a refinement.

Then, by \cite[Theorem 5.7]{faltings-chai-degenerations}, it follows that $\ov{\cA_g}^\Sigma$ is a smooth, separated Deligne-Mumford stack which is a simple normal crossings compactification of $\Ag$ and whose boundary complex is $\Lk \Agtrop{\Sigma}$.  Now the desired result follows from the following comparison theorem: we have a canonical isomorphism
$$\wt{H}_{i-1}(\Delta(\mathcal{X}\subset \ov{\mathcal{X}});\QQ) \cong \Gr^W_{2d} H^{2d-i} (\mathcal{X};\QQ),$$
for any normal crossings compactification $\mathcal{X}\subset \ov{\mathcal{X}}$ of smooth, separated Deligne-Mumford stacks over $\CC$, where $\Delta(\mathcal{X}\subset \ov{\mathcal{X}})$ denotes the boundary complex of the pair $(\mathcal{X}, \ov{\mathcal{X}})$  and $d = \dim \mathcal{X}$ is the complex dimension of $\mathcal{X}$. This comparison theorem follows from Deligne's mixed Hodge theory \cite{Deligne71, Deligne74b} in the case of complex varieties; we refer to \cite{cgp-graph-homology} for the generalization to Deligne-Mumford stacks.

\end{proof}

\remnow{\label{rem:msbj}
Let us briefly explain how Theorem~\ref{thm:correspondence} appears in \cite{odaka-tropical} and \cite{odaka-oshima-collapsing}, since the language of those papers is somewhat different.  Odaka and Odaka-Oshima study certain ``hybrid'' compactifications of arithmetic quotients $\Gamma\backslash D$ of Hermitian symmetric domains.  
The case of $\cA_g$  is the case $\Gamma = \mathrm{Sp}(2g,\ZZ)$ and $D$ is the ``unit disc''  of complex symmetric matrices $Z$  with $Z^t\ov{Z}<\mathrm{Id}_g.$  
The point is that the boundary of these compactifications is homeomorphic to $L\cA_g^{\rm{trop},\Sigma}$, so the comparison statement in \cite[Corollary 2.9]{odaka-oshima-collapsing}, which relies on \cite{cgp-graph-homology}, combined with Theorem 2.1 of op.~cit., reduces to Theorem~\ref{thm:correspondence} in this case.  

It is worth emphasizing the independence of choice of the admissible decomposition of $\Sigma$, as remarked in \cite[A.14]{odaka-tropical}, that was implicit in the discussion above. More precisely, for any two admissible decompositions $\Sigma_1$ and $\Sigma_2$ of $\PDrt_g$, we have a homeomorphism of links
\[\Lk \Agtrop{\Sigma_1} \cong \Lk \Agtrop{\Sigma_2}.\]
Indeed, 
it is well known that any two admissible decompositions $\Sigma_1$ and $\Sigma_2$ admit a common refinement $\widetilde{\Sigma}$ which is an admissible decomposition \cite[IV.2, p.~97]{faltings-chai-degenerations}, 
and by the construction of \S\ref{subsec:links}, we have canonical homeomorphisms $\Lk \Agtrop{\Sigma_1} \cong \Lk \Agtrop{\tilde\Sigma} \cong \Lk \Agtrop{\Sigma_2}.$
}

\section{The Perfect and Voronoi Chain Complexes}
\label{sec:chaincomplexes}

In computing the top-weight cohomology of $\cA_{g}$ there are two chain complexes that play central roles: the perfect chain complex $P^{(g)}_{\bullet}$ and the Voronoi chain complex $V^{(g)}_{\bullet}$. In this section we define both of these complexes, and show that the homology of the perfect chain complex $P^{(g)}_{\bullet}$ computes the top-weight cohomology of $\cA_{g}$. Further, we show that the perfect and Voronoi complexes are related via a short exact sequence of chain complexes, which is useful as the Voronoi complex has seen more extensive study \cite{elbaz-vincent-gangl-soule-perfect,sikiri2019voronoi}.  We make use of this short exact sequence to prove our main results in Section~\ref{sec:computations-of-cohomology}.

\subsection{The perfect chain complex}\label{S:PerfectChainComplex}
We first fix some notation, most of which we adapt from Section~\ref{subsec:links}. For $n\in \ZZ$, let $\Sigma_{g}^{\perf}[n]$ be the set of perfect cones in $\Sigma_{g}^{\perf}$ of dimension $n+1$, and denote the finite set of $\GL_{g}(\ZZ)$-orbits of such cones by $\Sigma_{g}^{\perf}[n]/\GL_{g}(\ZZ)$. We write $\sigma \sim  \sigma'$ if and only if $\sigma$ and $\sigma'$ lie in the same $\GL_{g}(\ZZ)$-orbit. Recall that a cone $\sigma\in \Sigma_{g}^{\perf}$ is {\em alternating} if and only if every element of $\GL_{g}(\ZZ)$ stabilizing $\sigma$ induces an orientation-preserving cone morphism of $\sigma$. If $\sigma$ is an alternating cone then every cone in the same $\GL_{g}(\ZZ)$-orbit as $\sigma$ is alternating. We call such $\GL_{g}(\ZZ)$-orbits alternating. Let $\Gamma_{n}^{(g)}=\Gamma_{n}$ be a set of representatives for the alternating elements of $\Sigma_{g}^{\perf}[n]/\GL_{g}(\ZZ)$. 

For each $n$ and each $\sigma \in \Gamma_{n}$, choose an orientation $\omega_{\sigma}$ on $\sigma$; the $\GL_g(\ZZ)$-action extends this choice to a choice of orientation on every alternating cone in $\Sigma^{\perf}_g$.  If $\rho \subset \sigma$ is an alternating facet of $\sigma$, denote the orientation induced on $\rho$ by $\omega_{\sigma}|_{\rho}$. Now let $\eta(\rho,\sigma)$ be 1 if the orientation on $\rho$ agrees with the orientation induced by $\sigma$ (i.e., $\omega_{\rho}=\omega_{\sigma}|_{\rho}$) and $-1$ otherwise. Finally, given $\sigma\in \Gamma_{n}$ and $\sigma'\in \Gamma_{n-1}$ define 
\begin{equation}\label{def:delta}
\delta(\sigma',\sigma)\coloneqq \sum_{\substack{\rho \subset \sigma \\ \rho \sim \sigma'}} \eta(\rho,\sigma)
\end{equation}
where the sum is over all facets $\rho$ of $\sigma$ in the same $\GL_{g}(\ZZ)$-orbit as $\sigma'$. With this notation in hand we can now define the perfect chain complex. 

\begin{definition}\label{defn:perfect-complex}
The \emph{perfect chain complex} ($P^{(g)}_{\bullet},\partial_{\bullet})$ is the rational complex defined as follows. For each $n$,  $P_{n}^{(g)}$ is the $\QQ$-vector space with basis indexed by $\Gamma_{n}$. The differential $\partial_{n}:P_{n}^{(g)}\to P^{(g)}_{n-1}$ is given by
\[
\partial(e_{\sigma}) \coloneqq \sum_{\sigma'\in \Gamma_{n-1}} \delta(\sigma',\sigma)e_{\sigma'}.
\]
\end{definition}

Notice that $P_{n}^{(g)}$ is only possibly nonzero in the range $-1\leq n \leq \binom{g+1}{2}-1$, but even within this range $P_{n}^{(g)}$ may be zero since alternating perfect cones do not necessarily exist in every dimension (see Example~\ref{ex:g2-perfect-complex}). While in many cases $\delta(\sigma',\sigma)$ is equal to $-1,0,$ or $1$, this need not always be the case since a cone may have two or more facets that are $GL_g(\mathbb{Z})$-equivalent.

\begin{example}\label{ex:g2-perfect-complex}
When $g=2,$ recall from Example~\ref{E:Per-g2} that up to the action of $\GL_{2}(\ZZ)$ there is precisely one cone of maximal dimension,
\[
\sigma_3\coloneqq \mathbb{R}_{\ge 0} \left \langle  \begin{pmatrix}1&0\\0&0\end{pmatrix}, \begin{pmatrix}0&0\\0&1 \end{pmatrix}, \begin{pmatrix} 1&-1\\-1&1\end{pmatrix}\right \rangle.\]
Since $\sigma_3$ is simplicial, its faces correspond to all subsets of the above ray generators. One can show that up to the action of $\GL_{2}(\ZZ)$ there is at most one cone in each dimension: 
$$
    \sigma_{2}\coloneqq \mathbb{R}_{\ge 0} \left \langle  \begin{pmatrix}1&0\\0&0\end{pmatrix}, \begin{pmatrix}0&0\\0&1 \end{pmatrix}\right \rangle,\ \ \
    \sigma_{1}\coloneqq \mathbb{R}_{\ge 0} \left \langle  \begin{pmatrix}1&0\\0&0\end{pmatrix}\right \rangle,\ \ \
    \sigma_{0}\coloneqq \mathbb{R}_{\ge 0} \left \langle  \begin{pmatrix}0&0\\0&0\end{pmatrix}\right \rangle.
$$

\begin{figure}[h]
\begin{center}
\includegraphics[height = 2 in]{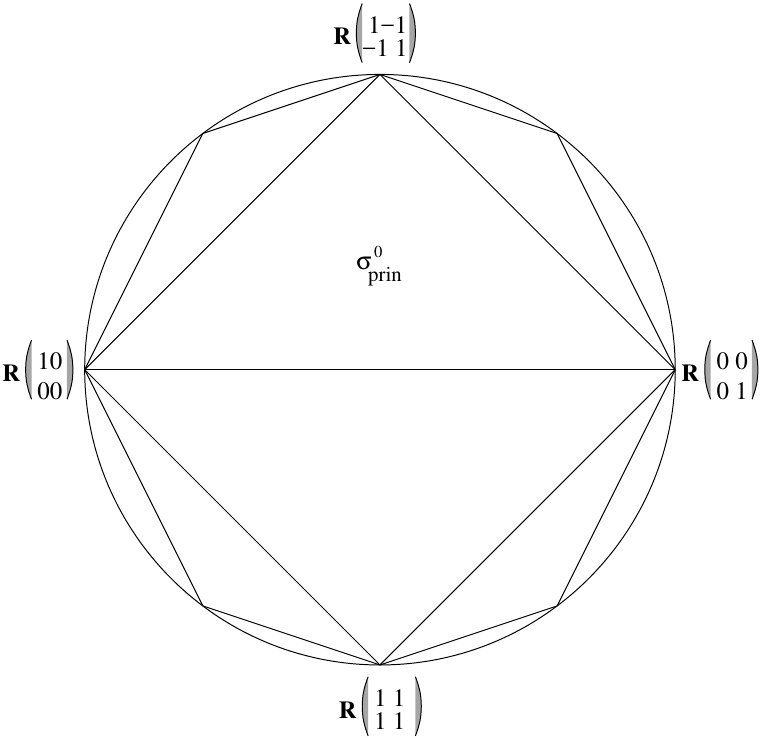}
\end{center}
\caption{A section of $\Omega_2^{\rm rt}$ and its perfect cone 
decomposition.}
\label{VorFig}
\end{figure}
Thus, to determine $\Gamma_{-1}$, $\Gamma_{0}$, $\Gamma_{1}$, and $\Gamma_{2}$, it is enough to see which of $\sigma_{0}$, $\sigma_{1}$, $\sigma_{2}$, and $\sigma_3$ are alternating. Consider the 
matrix
\[
A = \begin{pmatrix} 0&1\\1&0\end{pmatrix}.
\]
One can show that $A$ stabilizes both $\sigma_{2}$ and $\sigma_{3}$ and that the induced cone morphism is orientation-reversing. 
Thus, $\Gamma_{1}$ and $\Gamma_{2}$ are empty. On the other hand, since the action of $\GL_{2}(\ZZ)$ fixes the cone point $\sigma_{0}$, both $\sigma_{0}$ and $\sigma_{1}$ are alternating.  So, $\Gamma_{-1}=\{\sigma_{0}\}$ and $\Gamma_{0}=\{\sigma_{1}\}$.
From this we see that the complex $P^{(2)}_\bullet$ is:
\begin{center}
    \begin{tikzcd}[row sep = .5em, column sep = 3.5em]
    P^{(2)}_2 & P^{(2)}_1 & P^{(2)}_0 & P^{(2)}_{-1} & P^{(2)}_{-2} \\
    0 \rar & 0 \rar & \QQ \langle e_{\sigma_{1}}\rangle \rar{\partial_{0}} & \QQ \langle e_{\sigma_{0}}\rangle \rar & 0
    \end{tikzcd}
\end{center}
where $\partial_{0}$ sends $e_{\sigma_{1}}$ to either $e_{\sigma_{0}}$ or $-e_{\sigma_{0}}$ depending on the chosen orientations.
\end{example}

\begin{remark}
To be precise, the perfect complex $P^{(g)}_\bullet$ as constructed in Definition~\ref{defn:perfect-complex} is only unique up to isomorphism. In particular, the choice of representatives for $
\Gamma_{n}$ or reference orientations may result in different but isomorphic chain complexes. For instance, in Example~\ref{ex:g2-perfect-complex}, the differential $\partial_0$ is only determined up to sign.
\end{remark}

The next proposition shows that 
the perfect complex $P^{(g)}_\bullet$ is isomorphic to the cellular chain complex associated to the symmetric CW complex $\Lk \Agtrop{\perf}$. Thus, by  Theorem~\ref{thm:correspondence}, the homology of $P^{(g)}_\bullet$ computes the top-weight rational cohomology of $\cA_g$. 

\begin{proposition}\label{prop:pg-cellular-chain}
For each $i\ge 0$, there exist canonical isomorphisms
\[
H_{i-1}(P^{(g)}_{\bullet}) \cong \widetilde{H}_{i-1}(\Lk \Agtrop{\perf};\QQ)\cong \Gr^W_{2d}H^{2d-i}(\cA_g;\QQ).
\]
where $\Agtrop{\perf} = \Agtrop{\Sigma^{\perf}_g}$ is the tropical moduli space constructed in Section~\ref{subsec:Agtrop}.
\end{proposition}

\begin{proof}
By construction, $P^{(g)}_{\bullet}$ is naturally isomorphic to the cellular chain complex of $\Lk \Agtrop{\perf}$ as defined in Section~\ref{subsec:links}. Observe that the space $ \Agtrop{\perf}$ is connected since it deformation retracts to the cone point. Thus, the first isomorphism then follows from part (1) of Proposition~\ref{prop:cellular-complex} and the second isomorphism follows from Theorem \ref{thm:correspondence}.
\end{proof}



\begin{example}
By Example~\ref{ex:g2-perfect-complex}, we see that $P^{(2)}_\bullet$ has trivial homology in all degrees. Thus, Proposition~\ref{prop:pg-cellular-chain} recovers the fact that $\cA_{2}$ has trivial top-weight cohomology \cite{igusa-siegel}.
\end{example}

\subsection{The Voronoi Complex}\label{S:Voronoi}
Now we introduce a closely related complex, called the Voronoi complex $V_{\bullet}^{(g)}$, as considered in \cite{elbaz-vincent-gangl-soule-perfect,sikiri2019voronoi}\footnote{In \cite{elbaz-vincent-gangl-soule-perfect,sikiri2019voronoi} the Voronoi complex is defined as a complex of free $\ZZ$-modules, while our definition of Voronoi complex is as a complex of $\QQ$-vector spaces.}.
We shall soon see that $V^{(g)}_{\bullet}$ is a quotient of $P^{(g)}_{\bullet}$, obtained by setting to zero the generators corresponding to cones contained in the boundary of $\PDrt_{g}$.

For each $n\in \ZZ$, let $\overline{\Gamma}^{(g)}_{n} = \overline{\Gamma}_{n}$ be the subset of $\Gamma_{n}$ consisting of those cones $\sigma$ such that $\sigma \cap \PD_{g} \neq \emptyset$. For each $\sigma \in \overline{\Gamma}_{n}$, let $\omega_{\sigma}$ be an orientation of $\sigma$, and for each $\sigma\in\ov{\Gamma}_{n-1}$ define $\delta(\sigma',\sigma)$ as before in~\eqref{def:delta}.  With this notation, we can now define the Voronoi complex. 

\begin{definition}\label{def:VoronoiComplex}
The \emph{Voronoi chain complex} ($V^{(g)}_{\bullet},d_{\bullet})$ is the complex where $V_{n}^{(g)}$ is the $\QQ$-vector space with basis indexed by $\overline{\Gamma}_{n}$ and the differential $d_{n}:V_{n}^{(g)}\to V^{(g)}_{n-1}$ is given by
\[
d(e_{\sigma}) \coloneqq \sum_{\sigma'\in \overline{\Gamma}_{n-1}} \delta(\sigma',\sigma)e_{\sigma'}.
\]
\end{definition}

\begin{example}
The $\GL_{2}(\ZZ)$-orbits of alternating cones in $\Sigma^{\perf}_2$ are all contained in $\PDrt_{2}\setminus\PD_{2}$. 
Hence the Voronoi complex $V^{(2)}_{\bullet}$ is zero in all degrees. 
\end{example}

There is a natural surjection of chain complexes
$P_{\bullet}^{(g)} \twoheadrightarrow V_{\bullet}^{(g)}$
given by quotienting $P^{(g)}_n$ by the subcomplex spanned by those cones contained in $\PDrt_{g}\setminus\PD_{g}$. Our next goal, achieved in Theorem~\ref{thm:exact}, is to show that the kernel of the map above can naturally be identified with $P^{(g-1)}_{\bullet}$.  We begin by noting the following two lemmas studying those cones lying in $\PDrt_{g}\setminus\PD_{g}$.

\lemnow{\label{lem:when-boundary}Let 
$\sigma = \RR_{\ge 0} \langle v_1v_1^t,\dots, v_nv_n^t\rangle $ 
with $v_1,\dots,v_n \in \ZZ^g$ be a perfect cone. Then $\sigma$ is contained in $\PDrt_{g} \setminus \Omega_{g}$ if and only if $\dim \Span_{\RR}\langle v_{1},v_{2},\ldots,v_{n}\rangle<g$.}


\begin{lemma}\label{lem:big-cone-little-cone}
If $\sigma\in \Sigma^{\perf}_{g}$ is a perfect cone and $\sigma \subset \PDrt_{g}\setminus \PD_{g}$ then there is a matrix $A\in \GL_{g}(\ZZ)$ and a cone $\sigma'\in \Sigma_{g'}^{\perf}$,
where $g'<g$ and $\sigma'\cap \PD_{g'}\neq \emptyset$, with
\begin{equation} \label{eq:q000}
A \sigma A^t = \left\{ 
\begin{pmatrix}\begin{array}{@{}c|c@{}}
  Q'
  & 0 \\
\hline
  0 &
  0
\end{array} \end{pmatrix}
\bigg| Q' \in \sigma' \right\}. 
\end{equation}
In this situation, say $\sigma'$ is a {\em reduction} of $\sigma$.
\end{lemma}


We now show that, in a sense that we shall make precise, the action of $\GL_g(\mathbb{Z})$ on $\sigma$ does not depend on the ambient matrix size $g$. For example, given a cone $\sigma \in \Sigma^{\perf}_{g}$ and a reduction $\sigma' \in \Sigma^{\perf}_{g'}$ of $\sigma$, we will see that $\sigma$ is alternating if and only if $\sigma'$ is alternating.  
We begin with the following definition.  

\begin{definition}\label{def:cone-hom}
Given perfect cones $\sigma_{1},\sigma_{2}\in \Sigma^{\perf}_{g}$, let $\Hom_{\PDrt_{g}}(\sigma_{1},\sigma_{2})$ denote the set of morphisms $\rho:\sigma_{1}\to \sigma_{2}$ which are restrictions from the action of $\GL_{g}(\ZZ)$ on $\PDrt_{g}$:
\begin{center}
    \begin{tikzcd}[row sep = 4em, column sep = 4em]
        \PDrt_{g} \rar[two heads]{X \mapsto AXA^{t}} & \PDrt_{g} \\
        \sigma_{1} \uar[hook] \rar{\rho} & \sigma_{2} \uar[hook]
    \end{tikzcd}.
\end{center}
\end{definition}

The following two results concerning homomorphisms of cones contained in the boundary of $\Omega_g^{\mathrm{rt}}$ are standard and possibly well-known to experts. We include proofs here, however, as we are unaware of suitable references. 

\begin{proposition}\label{prop:g-agnostic}
If $\sigma_{1},\sigma_{2}\in \Sigma^{\perf}_{g}$ are perfect cones contained in $\PDrt_{g}\setminus \PD_{g}$ and $\sigma'_{1},\sigma'_{2}\in \Sigma^{\perf}_{g'}$ are reductions of $\sigma_{1}$ and $\sigma_{2}$ respectively, then there exists a bijection
\begin{center}
    \begin{tikzcd}[column sep = 3.5 em]
    \Hom_{\PDrt_{g}}(\sigma_{1},\sigma_{2}) \rar[leftrightarrow]{\sim} & \Hom_{\PDrt_{g'}}(\sigma'_{1},\sigma'_{2}).
    \end{tikzcd}
\end{center}
\end{proposition}

\begin{proof}
By Lemma~\ref{lem:big-cone-little-cone}, we may assume $\sigma_i$ are in the form of~\eqref{eq:q000}.  Then if $\rho'\in \Hom_{\PDrt_{g'}}(\sigma'_{1},\sigma'_{2})$ arises from the action of a matrix $A'\in \GL_{g'}(\ZZ)$ on $\PDrt_{g'}$, then extending it by a $(g-g')\times(g-g')$ identity matrix gives a matrix $A \in \GL_{g}(\ZZ)$ that induces a cone morphism $\rho:\sigma_{1}\to\sigma_{2}$.

In the other direction, suppose that $\rho \in \Hom_{\PDrt_{g}}(\sigma_{1},\sigma_{2})$ comes from the action of a matrix $A\in \GL_{g}(\ZZ)$ on $\PDrt_{g}$.  Write $\RR^{g'}$ for the coordinate subspace of $\RR^{g}$ of vectors in which the last $g-g'$ coordinates are zero.
Let $\sigma_1 = \RR_{\ge0} \langle v_1 v_1^t, \ldots, v_nv_n^t \rangle $.  By Lemma~\ref{lem:when-boundary}, the vectors $v_1,\ldots,v_n$ span $\RR^{g'}$.  Since $A\sigma_1 A^t = \sigma_2$, it follows again from Lemma~\ref{lem:when-boundary} that $Av_1,\ldots,Av_n$ also span $\RR^{g'}$; thus $A$ restricts to a map $A'\col \RR^{g'}\to \RR^{g'}$, with $A(\ZZ^{g'}) \subseteq \ZZ^{g'}$.  Similarly, $A^{-1}$ restricts to $(A')^{-1}\col \RR^{g'}\to \RR^{g'}$, and $(A')^{-1}(\ZZ^{g'}) \subseteq \ZZ^{g'}$.  Therefore $A'\in \GL_{g'}(\ZZ)$ is an invertible integer matrix, with $A'\sigma_1' (A')^t = \sigma_2'$.

Finally, a direct computation shows that these constructions are mutual inverses. 
\end{proof}

As a corollary of  Proposition~\ref{prop:g-agnostic}, the properties of being in the same $\GL_{g}(\ZZ)$-orbit and being alternating do not depend on $g$---that is, they are preserved by taking reductions. 

\begin{corollary}\label{cor:orbits-g-agnostic}
Two perfect cones $\sigma_{1},\sigma_{2} \subset \PDrt_{g}\setminus \PD_{g}$ are in the same $\GL_{g}(\ZZ)$-orbit if and only if there exists a $g'< g$ and reductions $\sigma'_{1}, \sigma'_{2}\in \Sigma_{g'}^{\perf}$ that are in the same $\GL_{g'}(\ZZ)$-orbit.
A perfect cone $\sigma\subset \PDrt_{g}\setminus \PD_{g}$ is alternating if and only if there exists a reduction $\sigma'\in \Sigma_{g'}^{\perf}$ which is alternating.
\end{corollary}

\begin{proof}
Two perfect cones $\sigma_{1}$ and $\sigma_{2}$ are in the same orbit if and only if $\Hom_{\PDrt_{g}}(\sigma_{1},\sigma_{2})$ is nonempty. Then the claim follows from Proposition~\ref{prop:g-agnostic}, since $\Hom_{\PDrt_{g}}(\sigma_{1},\sigma_{2})$ is nonempty if and only if $\Hom_{\PDrt_{g'}}(\sigma'_{1},\sigma'_{2})$ is nonempty. Similarly, the proof of Proposition~\ref{prop:g-agnostic}, applied to $\sigma = \sigma_1=\sigma_2$, shows that $\sigma $ has an orientation-reversing automorphism if and only if its reduction $\sigma'$ does. \end{proof}



 


Corollary~\ref{cor:orbits-g-agnostic} allows us to naturally identify the set of $\GL_{g}(\ZZ)$-orbits of alternating perfect cones in $\PDrt_{g}\setminus\PD_{g}$ with the set of $\GL_{g-1}(\ZZ)$-orbits of alternating perfect cones in $\PDrt_{g-1}$. Thus, we have the following theorem. 

\begin{theorem}
\label{thm:exact}
We have a short exact sequence of chain complexes
\begin{center}
    \begin{tikzcd}
    0\rar{}&P^{(g-1)}_{\bullet} \rar & P^{(g)}_{\bullet} \rar{\pi} & V^{(g)}_{\bullet} \rar & 0. 
    \end{tikzcd}
\end{center}
\end{theorem}

\begin{proof}
By construction, the kernel of $\pi:P^{(g)}_{n}\to V_{n}^{(g)}$ is generated by those basis vectors $e_{\sigma}$ where $\sigma\in \Gamma^{(g)}_{n}\setminus\overline{\Gamma}^{(g)}_{n}$.  (Recall that $\Gamma_n^{(g)}$ denotes a set of representatives of alternating $\GL_g(\ZZ)$-orbits of cones in $\Sigma^P_g$, and  $\overline{\Gamma}^{(g)}_{n}$ denotes the subset of those that meet $\Omega_g$.) By Corollary~\ref{cor:orbits-g-agnostic}, such cones are in bijection with elements of $\Gamma_{n}^{(g-1)}$. The differentials on $P^{(g)}_\bullet$ and $P^{(g-1)}_\bullet$ are defined in the same fashion, so the result follows.
\end{proof}

Theorem~\ref{thm:exact} reflects the stratification of $L A_g^{\mathrm{trop}}$ by the spaces $L \Omega_{g'}/\GL_{g'}(\ZZ)$, for $g'=1,\ldots,g$, which are rational classifying spaces for $\GL_{g'}(\ZZ)$; this is the underlying geometric reason that it is possible to relate the cohomology of $\cA_g$ to that of $\GL_{g'}(\ZZ)$, as we do here. 
This possible relationship was  suggested in the more general setting of arithmetic quotients of Hermitian symmetric domains in \cite[\S2.4]{odaka-oshima-collapsing}.

\medskip

\section{The inflation complex and the coloop complex}\label{subsec:Ig}

In this section, we define a subcomplex of $P^{(g)}_{\bullet}$, called the inflation complex $I^{(g)}_{\bullet}$. We shall show in Theorem~\ref{thm:Ig-acyclic} that $I^{(g)}_{\bullet}$ is acyclic.  This acyclicity result implies a vanishing result for $H_{k}(P^{(g)}_{\bullet})$ in low degrees, obtained in Corollary~\ref{cor:low-degree-vanishing}, and it is invoked in the computations in the next section for $g=6$ and $g=7$.  In Section~\ref{subsec:coloop}, we define an analogous subcomplex, the {\em coloop} complex, of the regular matroid complex, and prove an analogous  acyclicity result. The acyclicity of the coloop complex will not  be used in this paper, but should likely be useful for future study of the regular matroid complex.

\subsection{The inflation complex}

\begin{definition}\label{def:coloop}
\mbox{}
\enumnow{
\item
Let $S\subset \ZZ^g$ be a finite set.  Say $v\in S$ is a {\em $\ZZ^g$-coloop} of $S$ if $v$ is part of a $\mathbb{Z}$-basis $v, w_2, \ldots, w_g$ for $\mathbb{Z}^g$ such that any $w \in S \setminus\{v\}$ is in the $\mathbb{Z}$-linear span of $w_2, \ldots, w_g$.
Equivalently, $v$ is a $\ZZ^g$-coloop if, up to the action of $\GL_g(\mathbb{Z})$, we may write $v = (0,\ldots,0,1)$ and $w = (*, \ldots, *, 0)$ for all $w \in S\setminus \{v\}$. 
\item 
Now let $\sigma = \sigma[Q]$ be a perfect cone in $\Sigma_g^{\perf}$. Recall that the set $M(Q)$ of minimal vectors has the property that $v\in M(Q)$ if and only if $-v \in M(Q)$; let  $M'(Q) = \{v_1,\ldots,v_n\}$ be a choice of one of $\{v,-v\}$ for each $v \in M(Q)$. So 
$$\sigma = \mathbb{R}_{\geq 0}\langle v_1v_1^t, \ldots, v_nv_n^t\rangle.$$
We say $v \in M'(Q)$ is a \emph{coloop} of $\sigma$ if 
$v$ is a $\ZZ^g$-coloop in $M'(Q)$.  
}
\end{definition}

\begin{remark}\label{rem:zz-vs-matroidal-coloops}

The definition of a $\ZZ^g$-coloop is inspired by the notion of a coloop of a matroid (i.e., an element not belonging to any circuit).  Indeed, if $S\subset \ZZ^g$ is any finite set and $v\in S$, then $v$ being a $\ZZ^g$-coloop of $S$ implies that $v$ is a coloop of $S$, considered as vectors in $\RR^g$.  The converse does not hold: for example, let $(v_1,v_2) = ((0,1),(3,2))$.  Then $v_1$ and $v_2$ are coloops of the matroid $M(v_1,v_2)$ over $\RR$, but neither  is a $\ZZ^2$-coloop. 

On the other hand, we prove in Lemma~\ref{lem:a-coloop-is-a-coloop} that if $M$ is a regular matroid, then $M$ has a coloop if and only if a totally unimodular matrix $A$ representing $M$ has column vectors with a $\ZZ^g$-coloop, if and only if $\sigma(M)$ has a coloop in the sense of Definition~\ref{def:coloop}(2).

\end{remark}

\begin{example}\label{example:coloop}
Consider the quadratic form defined by the positive definite matrix
\[
Q=\begin{pmatrix}
 2 & \frac{1}{2} & 1 \\
 \frac{1}{2} & 1 & \frac{1}{2} \\
 1 & \frac{1}{2} & 1
\end{pmatrix}.
\]
The minimum of $Q$ on $\mathbb{Z}^{3}-\{0\}$ is 1 and 
\[
M(Q)=\{(0,\pm1,0), (0,0,\pm1), \pm(1,0,-1), \pm(0,1,-1)\}\subset \mathbb{R}^{3}.
\]
The corresponding perfect cone $\sigma[Q]$ has a coloop, in particular, letting
\[
A = \begin{pmatrix}
 0 & -1 & 0 \\
0 & -1 & 1 \\
-1 & 0 & 1
\end{pmatrix}
\]
we see that
\[
A
\begin{pmatrix}
 2 & \frac{1}{2} & 1 \\
 \frac{1}{2} & 1 & \frac{1}{2} \\
 1 & \frac{1}{2} & 1
\end{pmatrix}A^{t} = 
\begin{pmatrix}
 1 & \frac{1}{2} & 0 \\
 \frac{1}{2} & 1 & 0 \\
 0 & 0 & 1
 \end{pmatrix}
\]
and so $M(AQA^{t})$ is $\{(\pm1,0,0),(0,\pm1,0),$ $\pm(1,1,0),(0,0,\pm1)\}$. 

The cone $\sigma[Q]$ can also be realized as the matroidal cone $\sigma[M(G)]$ where $G$ is the graph below. The coloop corresponds to the bridge edge of $G$.  
\begin{figure}[h]
    \centering
    \includegraphics[height = 0.3 in]{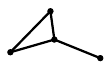} 
\end{figure}
\end{example}

\begin{lemma}\label{lem:study-2-coloops}
Let $S=\{v_1,\ldots,v_n\}\subset\ZZ^{g}$ and suppose that $v_1\ne v_2$ are both $\ZZ^g$-coloops for $S$.  Then there is a $\ZZ$-basis $v_1,v_2,w_3,\ldots,w_g$ for $\ZZ^g$ such that $$v_3,\ldots,v_n \in \ZZ\langle w_3,\ldots,w_g\rangle.$$
\end{lemma}
\begin{proof}
By restricting to $\RR\langle v_1,\ldots,v_n\rangle$, we may assume that $\RR\langle v_1,\ldots,v_n\rangle=\RR^g$.  Now the fact that both $v_1$ and $v_2$, being coloops, are in every basis for $\RR^g$ chosen from $\{v_1,\ldots,v_n\}$, implies that $v_3,\ldots,v_n$ span a $(g-2)$-dimensional subspace $V$ of $\RR^g$.
Let $w_3,\ldots,w_g$ be a $\ZZ$-basis for $V\cap \ZZ^g$.  We need only verify that $v_1,v_2,w_3,\ldots,w_g$ form a $\ZZ$-basis for $\ZZ^g$.  Let $x\in \ZZ^g$. Since $v_1,v_2,w_3,\ldots,w_g$ form a $\QQ$-basis, we have
$$x = a_1v_1 + a_2v_2 + a_3w_3 +\cdots+ a_gw_g, \quad\text{ for some }a_i\in \QQ,$$
and it suffices to show that $a_i \in \ZZ$ for all $i=1,2,3,\ldots,g$.  First, we show $a_1\in \ZZ$. Since $v_1$ is a $\ZZ^g$-coloop, after a change of $\ZZ$-basis, we may assume that $v_1 = (0,\ldots,0,1)$ and that
$v_2,\ldots,v_n$ have last coordinate zero.  Since $w_3,\ldots,w_g \in \Span_\RR\langle v_3,\ldots,v_n\rangle$, each $w_i$ also has last coordinate zero.  Therefore $a_1\in \ZZ$. By a similar argument, $a_2 \in \mathbb{Z}$. Then $a_3w_3 + \cdots + a_gw_g \in V\cap \ZZ^g$ for $a_i\in \QQ$.  But $w_3,\ldots,w_g$ is a $\ZZ$-basis for $V\cap \ZZ^g$. Therefore it must be that $a_3,\ldots,a_g\in \ZZ$, as desired.

\end{proof}

\cornow{\label{cor:once-a-coloop-always-a-coloop}
Let $S\subset \ZZ^{g-1}$ be a finite set, and identify $\ZZ^{g-1}$ with its image in $\ZZ^g$ under $w\mapsto (w,0)$.
Then $v$ is a coloop of $S$ if and only if $v$ is a coloop of $S\cup \{e_g\}\subset \ZZ^g$.  
}

\proofnow{The forward direction is direct from the definitions.  For the backward direction, suppose $v$ is a coloop of $S\cup \{e_g\}$.  Then {\em both} $v$ and $e_g$ are coloops, so by Lemma~\ref{lem:study-2-coloops}, up to the action of $\GL_g(\ZZ)$, we may assume that $v = e_{g-1}$ and $w = (*,\ldots,*,0,0)$ for all $w\in S\setminus\{v\}$.  Therefore $v$ was a coloop of $S\subset \ZZ^{g-1}$.}

\cornow{\label{cor:no-2-coloops}
A cone with two or more coloops is not alternating.  That is, if $\sigma=\sigma[Q]$ where $v\ne v'\in M'(Q)$ are two distinct coloops, then $\sigma$ is not alternating.
}

\begin{proof}
The cone $\sigma$ has an orientation-reversing automorphism induced by an element of $GL_g(\mathbb{Z})$ swapping the two coloops.
\end{proof}

We now describe two operations on cones, inflation and deflation, that add or remove coloops respectively. Inflation is described in \cite[Section 6.1]{elbaz-vincent-gangl-soule-perfect}, and can be performed for any cone, but we shall consider it for the cones in the set $\Sigma^P_{g,\mathrm{nco}}[n]$ defined below.

Recall that $\Sigma_g^{\perf}[n]$ denotes the set of $(n+1)$-dimensional perfect cones, and $\Sigma_g^{\perf}[n]/GL_g(\mathbb{Z})$ denotes the collection of $GL_g(\mathbb{Z})$-orbits of $(n+1)$-dimensional perfect cones. 

\begin{definition} 

We define two subsets of $\Sigma_g^{\perf}[n]$ as follows:
\begin{align*}
\Sigma_{g,\mathrm{nco}}^{\perf}[n] &\coloneqq \{\sigma \in \Sigma_g^{\perf}[n]: \textup{$\rank(\sigma) \le g-1$ and $\sigma$ has no coloop}\},\\
\Sigma_{g,\mathrm{co}}^{\perf}[n] &\coloneqq \{\sigma \in \Sigma_g^{\perf}[n]: \sigma \textup{ has exactly one coloop}\}.
\end{align*}
We then define $\Sigma_{g,\mathrm{nco}}^{\perf}[n]/GL_g(\mathbb{Z})$ and $ \Sigma_{g,\mathrm{co}}^{\perf}[n]/GL_g(\mathbb{Z})$ to be the collection of $GL_g(\mathbb{Z})$-orbits of the respective sets.
\end{definition}

We now define inflation and deflation as operations on $\Sigma_{g,\mathrm{nco}}^{\perf}[n]/GL_g(\mathbb{Z})$ and $ \Sigma_{g,\mathrm{co}}^{\perf}[n]/GL_g(\mathbb{Z})$, and we show these operations are well defined in Lemma \ref{lem:inflation_well_defined}.

\begin{definition} 
\textit{Inflation} is the map 
\begin{align*}
\mathrm{ifl}: \Sigma_{g,\mathrm{nco}}^{\perf}[n]/GL_g(\mathbb{Z}) &\longrightarrow \Sigma_{g,\mathrm{co}}^{\perf}[n+1]/GL_g(\mathbb{Z})
\end{align*} defined as follows.
Given an element of $\Sigma_{g,\mathrm{nco}}^{\perf}[n]/GL_g(\mathbb{Z})$, 
choose a representative 
$$\sigma=\mathbb{R}_{\geq 0}\langle w_1w_1^t, \ldots, w_kw_k^t\rangle,$$ so that the $g$-th entry of the $w_i$ are zero (see Lemma \ref{lem:big-cone-little-cone}). Let 
$$\tilde \sigma = \mathbb{R}_{\geq 0}\langle w_1w_1^t, \ldots, w_kw_k^t, e_ge_g^t\rangle.$$
Then set $\mathrm{ifl}([\sigma]) = [\tilde\sigma]$.
\end{definition}

\remnow{We check that inflation is well defined in Lemma~\ref{lem:inflation_well_defined}.  
However, we pause to point out that $\tilde\sigma$ is indeed a perfect cone, as noted in \cite[Section 6.1]{elbaz-vincent-gangl-soule-perfect}:
if $Q\in \PD_{g-1}$ is a positive definite quadratic form such that $\sigma[Q]=\mathbb{R}_{\geq 0}\langle \tilde{w}_1\tilde{w}_1^t, \ldots, \tilde{w}_k\tilde{w}_k^t\rangle$, where $\tilde{w}_i$ denotes the truncation of $w_i$ by the last entry,  
 then the inflation of $\sigma$ is the cone associated to the quadratic form 

$$
\tilde{Q} = \begin{pmatrix}\begin{array}{@{}c|c@{}}
  Q
  & 0 \\
\hline
  0 &
  m(Q)
\end{array} \end{pmatrix}
$$
where $m(Q)$ is the minimum value of $Q$ on $\ZZ^{g-1}\setminus\{0\}$.
Moreover, $\tilde\sigma$ has exactly one coloop by Corollary~\ref{cor:once-a-coloop-always-a-coloop} and the fact that $\sigma$ had no coloops.
}

\begin{example}
Continuing Example~\ref{example:coloop}, we see that if $Q'$ is the positive definite quadratic form
\[
Q' = \begin{pmatrix}
 1 & \frac{1}{2} & \frac{1}{2} \\
 \frac{1}{2} & 1 & \frac{1}{2} \\ 
 \frac{1}{2} & \frac{1}{2} & 2
\end{pmatrix},
\]
then $M(Q') = \{\pm(1,0,0),\pm(0,1,0), \pm(1,-1,0)\}$. Thus, the cone $\sigma[Q]$ is the inflation of $\sigma[Q']$. We may describe the cone $\sigma[Q']$ as
\[
\sigma[Q']=\left\{ 
\begin{pmatrix}\begin{array}{@{}c|c@{}}
  Q''
  & 0 \\
\hline
  0 &
  0
\end{array} \end{pmatrix}
\quad \bigg| \quad Q'' \in \sigma\left[\begin{pmatrix}
  1 & \frac{1}{2} \\ 
  \frac{1}{2} & 1
\end{pmatrix}\right] \right\},
\]
from which we see that $\sigma[Q']$ does not meet $\PD_{3}$, but its inflation $\sigma[Q]$ does. In general, the inflation of a perfect cone corresponding to a quadratic form of rank $r$ will itself be a perfect cone corresponding to a quadratic form of rank $r+1$.

\end{example}

\begin{definition}
We define the \textit{deflation} operation as a map
\begin{align*}
\mathrm{dfl}: \Sigma_{g,\mathrm{co}}^{\perf}[n+1]/GL_g(\mathbb{Z}) &\longrightarrow \Sigma_{g,\mathrm{nco}}^{\perf}[n]/GL_g(\mathbb{Z})
\end{align*}
given as follows.
Given an element of $ \Sigma_{g,\mathrm{co}}^{\perf}[n+1]/GL_g(\mathbb{Z})$, pick a $GL_g(\mathbb{Z})$-representative
\[\tilde\sigma=\mathbb{R}_{\ge 0}\langle w_1w_1^t,\dots,w_kw_k^t, e_ge_g^t\rangle,\]
where each $w_i$ is zero in the last coordinate.  Let $\sigma = \mathbb{R}_{\geq 0}\langle w_1w_1^t, \ldots, w_kw_k^t \rangle$.
It is routine to check that $\sigma$ really is a perfect cone, and moreover, it has no coloops by Corollary~\ref{cor:once-a-coloop-always-a-coloop}.
Then we set
$\mathrm{dfl}([\tilde\sigma]) = [\sigma].$ We now show that inflation and deflation are well defined.
\end{definition}

\begin{lemma}
\label{lem:inflation_well_defined}
For each $n \in \mathbb{N}$, inflation is a well-defined operation on $\Sigma_{g,\mathrm{nco}}^{\perf}[n]/GL_g(\mathbb{Z})$, and deflation is a well-defined operation on $\Sigma_{g,\mathrm{co}}^{\perf}[n+1]/GL_g(\mathbb{Z})$. Furthermore, these operations are inverses of each other.
\end{lemma}
\begin{proof}

We start with inflation. Given $[\sigma] \in \Sigma_{g,\mathrm{nco}}^{\perf}[n]/\GL_g(\mathbb{Z})$,
let $\sigma_1 = \mathbb{R}_{\geq 0}\langle v_1v_1^t, \ldots, v_kv_k^t\rangle$ and $\sigma_2 = \mathbb{R}_{\geq 0}\langle w_1w_1^t, \ldots, w_kw_k^t\rangle$ be two $\GL_g(\mathbb{Z})$-representatives of $[\sigma]$
such that the $g$-th entry of each of the $v_i,w_j$ is zero. By Proposition \ref{prop:g-agnostic}, there exist reductions $\sigma_1'$ of $\sigma_1$ and $\sigma_2'$ of $\sigma_2$ as well as an $A\in \GL_{g-1}(\mathbb{Z})$ sending $\sigma_1'$ to $\sigma_2'$ in $\PDrt_{g-1}$. 
Then 
$$
A'= \begin{pmatrix}\begin{array}{@{}c|c@{}}
  A
  & 0 \\
\hline
  0 &
  1
\end{array} \end{pmatrix}
$$
yields an equivalence between the two inflations. 

Now let $[\sigma] \in \Sigma_{g,\mathrm{co}}^{\perf}[n+1]/\GL_g(\mathbb{Z})$.
Let
\[\sigma_1 = \mathbb{R}_{\ge 0}\langle v_1v_1^t,\dots, v_kv_k^t,e_ge_g^t \rangle, \  \sigma_2 = \mathbb{R}_{\ge 0}\langle w_1w_1^t,\dots, w_kw_k^t,e_ge_g^t \rangle\]
be  two $\GL_g(\mathbb{Z})$ representatives of $[\sigma]$ such that the $v_i,w_j$ have $g$-th coordinate zero. Then by the proof of Proposition~\ref{prop:g-agnostic}, there exists an $A \in \GL_g(\mathbb{Z})$ 
such that 
$$Av_i = \pm w_i\text{ for }i=1,\ldots,k,\quad Ae_g = \pm e_g,$$
possibly after reordering the $v_i$. Indeed, $A$ must take the coloop $\pm e_g$ to $\pm e_g$.
Then $A$ gives an equivalence between the deflations $\mathbb{R}_{\ge 0} \langle v_1v_1^t, \dots, v_kv_k^t\rangle \sim \mathbb{R}_{\ge 0} \langle w_1w_1^t, \dots, w_kw_k^t\rangle$.

We now have that inflation and deflation are well defined, and it is clear from the definitions that these two operations are inverses. 
\end{proof}

\begin{lemma}
\label{lem:inflation_alt}
Let $\sigma = \mathbb{R}_{\geq 0}\langle v_1v_1^t, \ldots, v_kv_k^t\rangle$ be a perfect cone in $\Sigma_g^{\perf}$ of rank $<g$ with no coloop. Then $[\sigma]$ is alternating if and only if $\mathrm{ifl}([\sigma])$ is alternating.
\end{lemma}
\begin{proof}
We may assume that $v_1,\ldots,v_n$ have last coordinate $0$, so that, letting $$\tilde{\sigma} = \RR_{\ge0} \langle v_1v_1^t, \ldots, v_kv_k^t, e_g e_g^t\rangle,$$ we have $\mathrm{ifl}(\sigma) = \tilde{\sigma}.$
We claim there is a natural bijection \begin{equation}\label{eq:aut-aut} \Aut(\sigma) \longleftrightarrow \Aut(\tilde{\sigma}),\end{equation}
where $\Aut(\sigma)=\Hom_{\Omega^{\mathrm{rt}}_g}(\sigma,\sigma)$ (see Definition~\ref{def:cone-hom})  and similarly for $\Aut(\tilde \sigma)$.  Moreover, we claim that~\eqref{eq:aut-aut} takes orientation-preserving/reversing automorphisms of $\sigma$ to  orientation-preserving/reversing automorphisms of $\tilde\sigma$, respectively.

Given $\rho\in \Aut(\sigma)$ arising from a matrix $A\in \GL_{g-1}(\ZZ)$, the matrix 
$$\tilde A =\left(\begin{matrix}A &0\\0&1\end{matrix}\right)$$
yields an automorphism $\tilde{\rho}$ of $\tilde\sigma$.   The linear span of $\tilde \sigma$ is the sum of the linear span of $\sigma$ and that of $e_ge_g^t$.  Moreover, $\tilde{\rho}$ fixes the ray $e_ge_g^t$ of $\tilde\sigma$ and acts on the linear span of $\sigma$ according to $A$; in particular, $\tilde{\rho}$ is orientation-preserving if and only if $\rho$ was.

Next, suppose $\tilde A\in \GL_g(\ZZ)$ induces $\tilde\rho \in \Aut(\tilde \sigma)$.  Recall that $e_g$ is the only coloop of $\tilde{\sigma}$, by  Corollary~\ref{cor:once-a-coloop-always-a-coloop}.
Therefore $Ae_g = \pm e_g$, and hence $A$ induces an automorphism of $\sigma$.  Finally, it is routine to check that the maps constructed between $\Aut(\sigma)$ and $\Aut(\tilde\sigma)$ are  two-sided inverses. 

\end{proof}

\begin{definition}\label{def:Ig}
Let $I^{(g)}_{\bullet}$ be the subcomplex of $P^{(g)}_{\bullet}$ generated in degree $n$ by cones $\sigma \in \Gamma_n$ of rank $\leq g-1$ and cones of rank $g$ with a coloop. 
\end{definition}

\begin{theorem}\label{thm:Ig-acyclic}
The chain complex $I^{(g)}_{\bullet}$ is acyclic. 
\end{theorem}

\begin{proof}
By Lemmas \ref{lem:inflation_well_defined} and \ref{lem:inflation_alt} there is a matching of cones generating $I^{(g)}_{\bullet}$, given by 
$$
        \sigma \rightarrow \begin{cases}
                      \mathrm{ifl}(\sigma) &\text{ if $\sigma$ has no coloop}, \\
                        \mathrm{dfl}(\sigma) &\text{ if $\sigma$ has a coloop}.
                    \end{cases}
$$
Here, we have abused notation slightly, since $\mathrm{ifl}$ is an operation on orbits rather than orbit representatives.  Thus, when we write $\mathrm{ifl}(\sigma)=\sigma'$ for $\sigma\in \Gamma_n$ we mean that $\sigma'$ is the unique orbit representative in $\Gamma_{n+1}$ such that $\mathrm{ifl}([\sigma])=[\sigma']$. Similarly for deflation.

Now let $\sigma$ be a generator in $I^{(g)}_\bullet$ of maximal degree; then $\sigma$ must have a coloop $v$. We claim that $\sigma' = \mathrm{dfl}(\sigma)$ is not a facet of any other generator $\tau \not = \sigma$ of $I^{(g)}_\bullet$.  Indeed, suppose
$$\tau = \RR_{\ge 0} \langle v_1v_1^t, \ldots, v_kv_k^t\rangle$$
is a generator of $I^{(g)}_\bullet$ containing $\sigma'$ as a facet.  
If  $\tau$ had a coloop, say $v_n$, then since $\sigma'$ has no coloop, $\sigma'$ must not contain the ray $v_nv_n^t$.  But $$\RR_{\ge 0} \langle v_1v_1^t, \ldots, v_{k-1}v_{k-1}^t\rangle$$ is already a facet of $\tau$, so it must be $\sigma'=\mathrm{dfl}(\sigma) = \mathrm{dfl}(\tau),$ implying $\tau=\sigma$.  So $\tau$ has no coloop.  But then $\mathrm{ifl}(\tau)$ is a generator of $I^{(g)}_{\bullet}$ and it would have higher rank than $\sigma$.

Thus, the complex ${I^{(g)}_{\bullet}}'$ spanned by all cones except $\sigma$ and $\mathrm{dfl}(\sigma)$ is a subcomplex. Then we have a short exact sequence 
\[
\begin{tikzcd}
0 \rar{} &  {I^{(g)}_{\bullet}}'  \rar{} &  I^{(g)}_{\bullet}  \rar{} &  I^{(g)}_{\bullet} / {I^{(g)}_{\bullet}}'  \rar{} &  0,
\end{tikzcd}
\]
where $I^{(g)}_{\bullet} / {I^{(g)}_{\bullet}}'$ is isomorphic to $0 \rightarrow \sigma \rightarrow \mathrm{dfl(\sigma)} \rightarrow 0$. Hence ${I^{(g)}_{\bullet}}' \rightarrow I^{(g)}_{\bullet}$ is a quasi-isomorphism. Repeating this, we deduce inductively that $I$ is quasi-isomorphic to 0.
\end{proof}

As a corollary of Theorem~\ref{thm:Ig-acyclic}, we are able to prove the following vanishing result for the cohomology of $P^{(g)}_{\bullet}$ in low degrees.

\begin{corollary}\label{cor:low-degree-vanishing}
If $k\leq g-2$ then $H_{k}\left(P^{(g)}_{\bullet}\right)=0$.
\end{corollary}

\begin{proof}
Since the inflation complex $I^{(g)}_{\bullet}$ is acyclic by  Theorem~\ref{thm:Ig-acyclic}, it is enough to show that $I^{(g)}_{k}=P^{(g)}_{k}$ for all $k\leq g-2$.   For this, we simply need the well-known fact that the rank of a perfect cone is at most its dimension.  Indeed, let $\sigma=\RR_{\geq0} \langle v_{1}v_{1}^{t},v_{2}v_{2}^{t},\ldots,v_{n}v_{n}^{t}\rangle \in \Sigma^{\perf}_{g}$ be an alternating cone of dimension $k+1$. If $Q\in \sigma$ then \[
Q = \lambda_{1}v_{i_{1}}v_{i_{1}}^{t}+\lambda_{2}v_{i_{2}}v_{i_{2}}^{t}+\ldots+\lambda_{k+1}v_{i_{k+1}}v_{i_{k+1}}^{t}
\]
for some $\{i_1,\ldots,i_{k+1}\} \subset \{1,\ldots,n\}$ and some $\lambda_{i}\in \RR_{\ge 0}$. In particular, since $v_{j}v_{j}^{t}$ is a rank one quadratic form, this implies that the rank of $Q$ is at most $k+1$. 
Thus, if $k+1 \le g-1$ then $\on{rank}(\sigma)<g$, implying that the orbit of $\sigma$ represents an element of $I^{(g)}_{k}$. 
\end{proof}

\begin{remark}
The inflation operation is, of course, a special case of taking the block sum of two perfect cones.  In this way one obtains a product map on chain complexes $P^{(g_1)} \otimes P^{(g_2)} \to P^{(g_1+g_2)},$ 
and a corresponding product on homology. 
This is reminiscent of the result of \cite{ght18} describing the stable cohomology of the matroidal partial compactifications $\ov{\cA}_g^{\mathrm{matr}}$ via $1$-sums of irreducible regular matroids.
Perhaps if one had nonvanishing statements for the latter product, then the cohomology classes detected in this paper could be used to construct infinite families of top-weight classes in $\mathcal{A}_g$.
\end{remark}

\begin{remark}
The proof of Corollary~\ref{cor:low-degree-vanishing} shows that any cone of dimension less than or equal to $g-1$ does not intersect $\PD_{g}$. This implies that $H_{k}(V^{(g)}_{\bullet})=0$ for all $k\leq g-2$. \end{remark}  

\begin{remark}\label{rem:vcd}
The virtual cohomological dimension of $\cA_g$ is $$\on{vcd}(\cA_g) = \on{vcd}(\on{Sp}(2g,\ZZ)) = g^2,$$ by \cite{borel-serre-corners}, see \cite{chen-looijenga-stable}.  In particular,
\[\Gr^W_{g^2+g} H^i(\cA_g;\QQ) = 0\quad \text{for all }i> g^2,\]
which is equivalent to, setting $i= 2\dim(\cA_g) -j-1 = g^2+g-j-1$, that $$H_{j}(P^{(g)})=0 \quad\text{for all } j<g-1.$$
Corollary~\ref{cor:low-degree-vanishing} thus reproves, in a completely different way, the vanishing in top weight of rational cohomology of  $\cA_g$ in degree above the virtual cohomological dimension.
\end{remark}

\subsection{The Regular Matroid Complex and Inflation}\label{subsec:coloop}
In this section, we introduce two combinatorially defined subcomplexes, $R^{(g)}_{\bullet}$ and $C^{(g)}_{\bullet}$, of $P^{(g)}_{\bullet}$ coming from regular matroids and regular matroids with coloops respectively.  These are not used further in this paper.  Nevertheless, the matroidal cones in $\Sigma^P_g$ have geometric significance: Alexeev and Brunyate, in proving the existence of a compactified Torelli map $\ov{\mathcal{M}_g}\to \ov{\mathcal{A}_g}^{\mathrm{perf}}$, conjectured an open locus on which $\ov{\mathcal{A}_g}^{\mathrm{perf}}$ and $\ov{\mathcal{A}_g}^{\mathrm{Vor}}$  are isomorphic and on which the two Torelli maps $\ov{\mathcal{M}_g}\to \ov{\mathcal{A}_g}^{\mathrm{perf}}$ and $\ov{\mathcal{M}_g}\to \ov{\mathcal{A}_g}^{\mathrm{Vor}}$ agree \cite{alexeev-brunyate}.  The fourth author and Viviani verified their conjecture, showing that the {\em matroidal partial compactification} $\mathcal{A}_g^{\mathrm{matr}}$, whose strata correspond to cones arising from regular matroids, is the largest such open subset \cite{melo-viviani-comparing}.  For possible future use in studying $R^{(g)}$, we establish in this section that the complex $C^{(g)}_{\bullet}$, which is a matroid analogue of $I^{(g)}_{\bullet}$, is acyclic.  

Given a cone $\sigma \in \Sigma_{g}^{\perf}[n]$, we say that $\sigma$ is a matroidal cone if and only if there exists a simple, regular matroid $M$ of rank at most $g$ such that $[\sigma]=\sigma(M)$ where $\sigma(M)$ is as described in Construction~\ref{consrt:matroid-cones}. Matroidal cones are simplicial.  Since the faces of a matroidal cone are themselves matroidal cones, the set of representatives of alternating cones arising from simple, regular matroids forms a subcomplex of $P^{(g)}$.

\begin{definition}
The \emph{regular matroid complex} $R^{(g)}_{\bullet}$ is the subcomplex of $P^{(g)}_{\bullet}$ generated in degree $n$ by cones $\sigma \in \Gamma_{n}$ such that $\sigma$ is a matroidal cone. 
\end{definition}

\begin{remark}
When $g=2$ and $g=3$, the complexes $R^{(g)}_{\bullet}$ and $P^{(g)}_{\bullet}$ are in fact equal. It would be interesting to understand in general how much larger $P^{(g)}_{\bullet}$ is compared to $R^{(g)}_{\bullet}$.
\end{remark}

Recall that an element $e$ of a matroid $M$ is a \emph{coloop} if it does not belong to any of the circuits of $M$; equivalently, $e$ is a coloop if it belongs to every base of $M$. When $M$ is a regular matroid, this is equivalent to the existence of a totally unimodular matrix  $A=[v_{1},v_{2},\ldots,v_{n}]$ representing $M$ such that   $v_{i}=(*,*,\ldots,*,0)$ for $i=1,2,\ldots,n-1$ and $e = v_{n}=(0,0,\ldots,0,1)$. It is worth establishing that the notions of a matroid coloop and a $\ZZ^{g}$-coloop agree for matroidal cones, as we show in the next lemma.  

\begin{lemma}\label{lem:a-coloop-is-a-coloop}
Let $M$ be a simple, regular matroid of rank $\leq g$. The cone $\sigma(M)$ has a $\ZZ^g$-coloop if and only if the matroid $M$ has a coloop.
\end{lemma}

\begin{proof}
Suppose that $\sigma(M)$ has a $\ZZ^{g}$-coloop. By definition, there exists a quadratic form $Q\in \PD_{g}$ such $[\sigma(Q)]=\sigma(M)$ and $M'(Q)=\{v_{1},v_{2},\ldots,v_{n}\}$ where $v_{i}=(*,*,\ldots,*,0)$ for $i=1,\ldots,n-1$ and $v_{n}=(0,0,\ldots,0,1)$. 
Then by the construction of $\sigma(M)$, the matrix $A=[v_{1},v_{2},\ldots,v_{n}]$ is a totally unimodular matrix representing $M$ over $\RR$. Therefore $v_{n}$ is a coloop of the matroid $M$. 

For the other direction, suppose that the regular matroid $M$ on the ground set $\{1,\ldots,n\}$ is represented by a full-rank totally unimodular $g'\times n$ matrix $A=[v_{1},v_{2},\ldots,v_{n}]$, for some $g'\le g$, and that $n$ is a coloop of $M$.  Then $n$ is in every base of $M$, so $v_n$ is in every full-rank $g'\times g'$ submatrix of $A$. Reorder so that the rightmost $g'\times g'$ submatrix is full rank; call it $B$.  Then $B\in \GL_{g'}(\ZZ)$ by total unimodularity of $A$. Consider the matrix $B^{-1}A$, which still represents $M$.  The rightmost $g'\times g'$ submatrix of $B^{-1}A$ is the identity. Moreover, each of the first $n-g'$ columns is of the form $(*,\ldots,*,0)$, for otherwise it could replace the last column in the rightmost square submatrix to form a full rank square matrix, contradicting that $n$ was a coloop.  This shows that $v_n$ is a $\ZZ^{g'}$-coloop of $v_1,\ldots,v_n \in \ZZ^{g'}$, and after padding by zeroes, $v_n$ is a $\ZZ^g$-coloop of the of $v_1,\ldots,v_n$.
\end{proof}

\begin{definition}
The \emph{coloop complex} $C^{(g)}_{\bullet}$ is the subcomplex of $P^{(g)}_{\bullet}$ generated in degree $n$ by cones $\sigma \in \Gamma_{n}$ such that $\sigma$ is a matroidal cone and either:
\begin{enumerate}
    \item the rank of $\sigma$ is $<g$, 
    \item the rank of $\sigma$ is equal to $g$ and $\sigma$ has one coloop.
\end{enumerate}
\end{definition}

By Lemma~\ref{lem:a-coloop-is-a-coloop}, the generators for $C^{(g)}_\bullet$ are the generators of $R^{(g)}_\bullet$ that are also generators of $I^{(g)}_\bullet$; in summary, we have inclusions of complexes 
\begin{center}
    \begin{tikzcd}[row sep = 3em, column sep = 3em]
        C^{(g)}_\bullet \rar[hook] \dar[hook] & R^{(g)}_\bullet \dar[hook]\\
         I^{(g)}_\bullet \rar[hook] & P^{(g)}_\bullet
    \end{tikzcd}.
\end{center}

Similar to the inflation complex, the coloop complex is acyclic. 

\begin{theorem}\label{thm:Cg-acyclic}
The chain complex $C^{(g)}_{\bullet}$ is acyclic. 
\end{theorem}

We omit the details of the proof of Theorem~\ref{thm:Cg-acyclic}. It is closely analogous to the proof of Theorem~\ref{thm:Ig-acyclic}, the key step being the following lemma.

\begin{lemma}
There is a bijection of sets between
\begin{center}
    \begin{tikzcd}[column sep = 4em]
    \left\{\begin{matrix}\text{alternating, regular matroids}\\ 
    \text{of rank $<g$ with 0 coloops}
    \end{matrix} \right\} \rar[leftrightarrow]{\sim} & 
    \left\{\begin{matrix}\text{alternating, regular matroids}\\ 
    \text{of rank $\leq g$ with 1 coloop}
    \end{matrix} \right\}.
    \end{tikzcd}
\end{center}
\end{lemma}

\medskip

\section{Computations on the cohomology of \texorpdfstring{$\mathcal{A}_{g}$}{Ag}}\label{sec:computations-of-cohomology}
In this section, we compute the top-weight cohomology of $\mathcal{A}_{g}$ for $3\leq g \leq 7$, proving Theorem~\ref{thm:main}. When $g=3,4,$ and $5$, we do this by studying the cones of $\Sigma_g^{\perf}$ arising from matroids, from which we explicitly compute the chain complex $P^{(g)}_{\bullet}$. We handle the cases when $g=6$ and $g=7$ by utilizing the long exact sequence in homology arising from Theorem~\ref{thm:exact}, as well as the fact that the inflation subcomplex $I^{(g)}_{\bullet}$ is acyclic (see Theorem~\ref{thm:Ig-acyclic}). 
Additionally, we prove a vanishing result for the top-weight cohomology of $\mathcal{A}_{g}$ for $g=8,9,$ and $10$ in Theorem~\ref{thm:top-weight-vanishing-g8910}. 

\subsection{The complex \texorpdfstring{$P_\bu^{(3)}$}{P3}}

For $g=3$, the fact that every perfect cone is matroidal allows us to compute the complex $P^{(3)}_{\bullet}$ directly. Using this description of $P^{(3)}_{\bullet}$, we then compute the top-weight cohomology of $\mathcal{A}_{3}$.

\begin{proposition}
The chain complex $P_\bu^{(3)}$ is
\begin{center}
    \begin{tikzcd}[row sep = .5em, column sep = 1.5em]
    & P_5^{(3)} & P_4^{(3)} & P_3^{(3)} & P_2^{(3)} & P_1^{(3)} & P_0^{(3)} & P_{-1}^{(3)} \\
    0 \rar & \QQ \rar & 0 \rar & 0 \rar & 0 \rar & 0 \rar & \QQ \rar{\sim} & \QQ \rar & 0. \\
    \end{tikzcd}
\end{center}

\end{proposition}

\begin{proof}
 The only top-dimensional perfect cone of $\Sigma_3^{\perf}/GL_3(\mathbb{Z})$ is the principal cone $\sigma_3^{\mathrm{prin}}$ coming from the complete graph $K_4$ \cite[p. 151]{voronoi-nouvelles}. 
 The principal cone $\sigma_3^{\mathrm{prin}}$ is alternating because the automorphisms of $K_4$ are all alternating permutations of its edges, and every automorphism of $\sigma_3^{\mathrm{prin}}$ arises from $\Aut(K_4)$ by Remark~\ref{rem:autos}.
 Thus, we have $P^{(3)}_5 \cong \mathbb{Q}$. 

The automorphisms on the codimension $i$ faces of $\sigma_3^{\mathrm{prin}}$ arise from matroids of graphs obtained from $K_4$ by deleting $i$ edges, see Figure \ref{fig:k4minusedges}. 
For $i=1,\ldots,4$, each of the matroids associated to graphs with $i$ edges removed from $K_4$ has an automorphism given by an odd permutation of the edges; see Figure \ref{fig:k4minus} for an example.
So we have $P^{(3)}_j = 0$ for $1 \leq j \leq 4$. The single ray and vertex of $\Sigma_3^{\perf}/GL_3(\mathbb{Z})$ are alternating, so $P^{(3)}_j \cong \mathbb{Q}$ for $j=0$ and $j=-1$.

\begin{figure}[h]
    \centering
    \begin{tabular}{c c c c c c c c c}
        \includegraphics[height = 0.3 in]{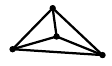} &
        \includegraphics[height = 0.3 in]{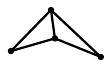} & 
          \includegraphics[height = 0.3 in]{k4graphs3.pdf} &
            \includegraphics[height = 0.3 in]{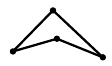} &
              \includegraphics[height = 0.3 in]{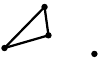} &
                \includegraphics[height = 0.3 in]{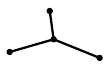} &
                  \includegraphics[height = 0.3 in]{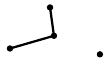} &
                    \includegraphics[height = 0.3 in]{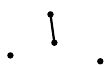} &
                    \includegraphics[height = 0.3 in]{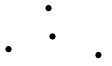}  \\
                    0 & 1 & 2 & 2 & 3 & 3 & 4 & 5 & 6
    \end{tabular}
    \caption{Graphs
    obtained by deleting the indicated number of edges from $K_4$, giving isomorphism classes of graphic matroids.}
    \label{fig:k4minusedges}
\end{figure}

\begin{figure}[h]
    \centering
    \includegraphics[height = 0.7 in]{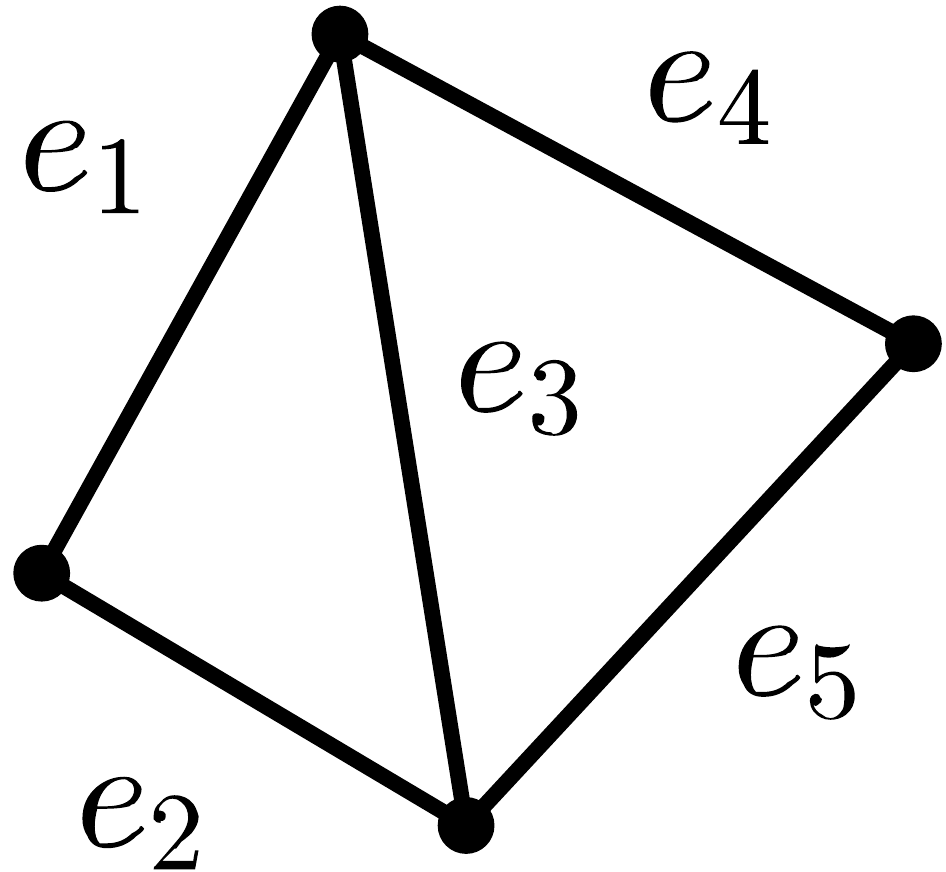}
    \hspace{0.5 in}
    \includegraphics[height = 0.7 in]{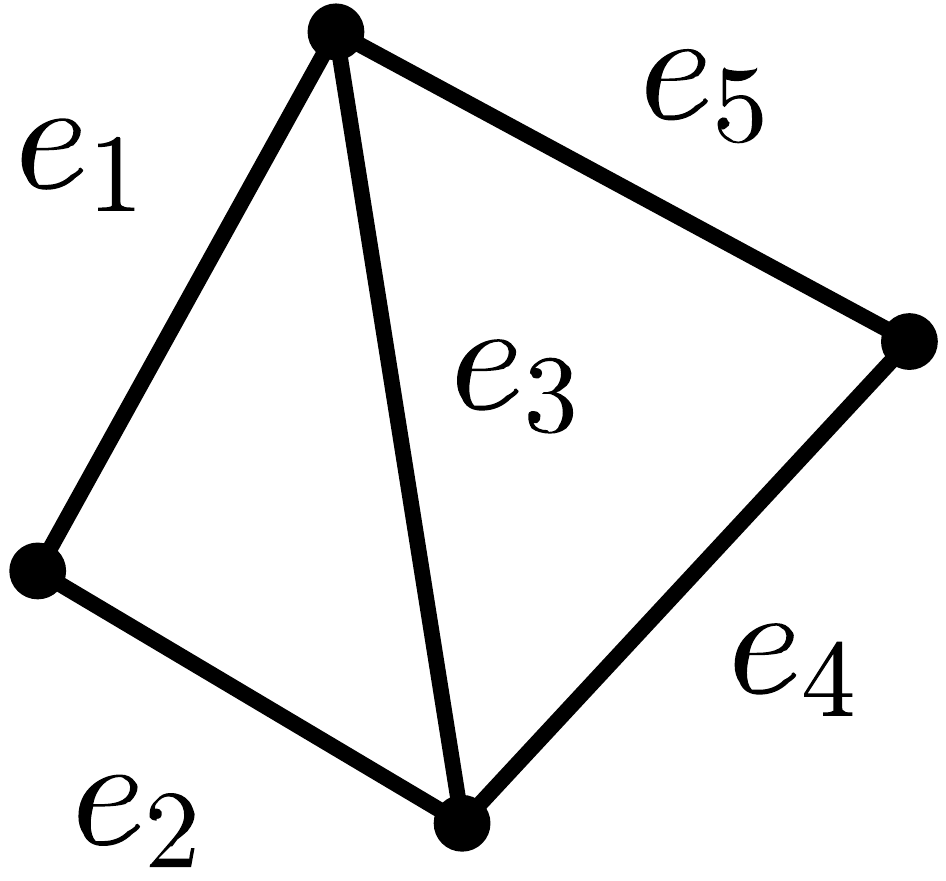}
    \caption{An automorphism of $M[K_4\backslash\{e_6\}]$ interchanging $e_4$ and $e_5$.}
    \label{fig:k4minus}
\end{figure}

\end{proof}

\begin{theorem}\label{thm:a3}
The top-weight cohomology of $\mathcal{A}_3$ is 
\[
\Gr_{12}^WH^i(\mathcal{A}_3;\mathbb{Q})=
\begin{cases}
\mathbb{Q} \text{ if }i=6\\
0\text{ else.}
\end{cases}
\]
\end{theorem}

\proofnow{The top-weight cohomology of $\mathcal{A}_3$ is the homology of $P_\bu^{(3)}$ by Theorem \ref{thm:correspondence}.}

\begin{remark}\label{rem:hain}
Theorem~\ref{thm:a3} agrees with the work of Hain \cite{hain-rational}, who computes the full cohomology ring of $\mathcal{A}_3$. Hain deduces in particular
$H^6(\cA_3;\QQ)  = E$
where $E$ is a mixed Hodge structure that is an extension
$0\to \QQ(-3)\to E \to \QQ(-6)\to 0,$
where $\QQ(n)$ denotes the Tate Hodge structure of dimension $1$ and weight $-2n.$
\end{remark}

\begin{example}
While we do not need it here, we note that using the fact that all of the perfect cones in $\Sigma^{\perf}_{3}$ arise from graphic matroids, one can check that the inflation complex $I^{(3)}_{\bullet}$ is the following:
\begin{center}
    \begin{tikzcd}[row sep = .5em, column sep = 1.5em]
    & I_5^{(3)} & I_4^{(3)} & I_3^{(3)} & I_2^{(3)} & I_1^{(3)} & I_0^{(3)} & I_{-1}^{(3)} \\
    0 \rar & 0 \rar & 0 \rar & 0 \rar & 0 \rar & 0 \rar & \QQ \rar{\sim} & \QQ \rar & 0. \\
    \end{tikzcd}
\end{center}
\end{example}

\subsection{The complex \texorpdfstring{$P_\bu^{(4)}$}{P4}} 

In this section, we explicitly compute the complex $P_\bu^{(4)}$ by using the matroidal description of the principal cone given in Section~\ref{subsec:matroids} together with the description of a similar complex for $\SL_{g}(\ZZ)$-alternating cones described in \cite{ls78}. We then use $P_\bu^{(4)}$ to compute the top-weight cohomology of $\cA_{4}$. 

\begin{proposition}
The chain complex $P_\bu^{(4)}$ is 
\begin{center}
    \begin{tikzcd}[row sep = .5em, column sep = .7em]
    P_9^{(4)} & P_8^{(4)} & P_7^{(4)} & P_6^{(4)} & P_5^{(4)} & P_4^{(4)} & P_3^{(4)} & P_2^{(4)} & P_1^{(4)} & P_0^{(4)} & P_{-1}^{(4)} \\
    0 \rar & 0 \rar & 0 \rar & \QQ \rar{\sim} & \QQ \rar & 0 \rar & 0 \rar & 0 \rar & 0 \rar & \QQ \rar{\sim} & \QQ \rar & 0. \\
    \end{tikzcd}
\end{center}
\end{proposition}

\begin{proof}
By Theorem \ref{thm:exact}, we have, in any degree $\ell$,  
that $\dim P_{\ell}^{(4)} =\dim P_{\ell}^{(3)} +\dim V_{\ell}^{(4)}.$
We have already computed $P_\bu^{(3)}$, so we now compute $ V_\bu^{(4)}$. 
In \cite{ls78}, the authors compute a complex $C_\bu$ which is generated in degree $i$ by the $(i+1)$-dimensional $\SL_4(\mathbb{Z})$-alternating perfect cones meeting $\Omega_g$ 
up to $\SL_4(\mathbb{Z})$-equivalence. Their results \cite[Proposition 3.1]{ls78} are summarized in the first three columns of Table \ref{tab:p4cones}.
The cone $\sigma(D_4)$ is the cone corresponding to the quadratic form 
$$
D_4 = \begin{bmatrix}
1 & 0 & 1/2 & 1/2 \\
0 & 1 & 1/2 & 1/2 \\
1/2 & 1/2 & 1 & 1/2 \\
1/2 & 1/2 & 1/2 & 1
\end{bmatrix}.
$$
For $i \not \in \{4,5,6,8,9\}$, they show that $C_i = 0$.

\begin{table}[h]
    \centering
\begin{tabular}{|c|c | c | c|}
\hline
   $i$  &  $C_i$ & $\SL_4(\mathbb{Z})$-alternating cones of dim $i$+1 & $\GL_4(\mathbb{Z})$-alternating?  \\
   \hline
    4 & $\mathbb{Q}$ & $\sigma(\includegraphics[height=0.14 in]{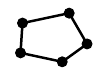})$ & no \\ 
    5 & $\mathbb{Q}$ & $\sigma(\includegraphics[height=0.14 in]{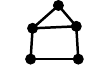})$ & no \\
    6 & $\mathbb{Q}$ & $\sigma(\includegraphics[height=0.14 in]{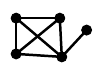})$ & yes \\
    8 & $\mathbb{Q}$ & $\sigma(\includegraphics[height=0.14 in]{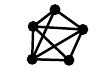})$ & no \\
    9 & $\mathbb{Q}^2$ & $\sigma_4^{\mathrm{prin}}$, $\sigma(D_4)$ & no \\
\hline
\end{tabular}
    \caption{$SL_4(\mathbb{Z})$-alternating cones of $\Sigma_4^{P} / SL_4(\mathbb{Z})$.}
    \label{tab:p4cones}
\end{table}

We now compute the Voronoi complex $V_\bu^{(4)}$. As far as we know, this computation---for $\GL_4(\ZZ)$, as opposed to $\SL_4(\ZZ)$---constitutes a small gap in the literature, which we fill here.  
To obtain the complex $V_\bu^{(4)}$, we must pass from $\SL_4(\mathbb{Z})$ to $\GL_4(\mathbb{Z})$. In doing so, two things may happen. First, two $\SL_4(\mathbb{Z})$-inequivalent cones may be $\GL_4(\mathbb{Z})$-equivalent. 
This does not occur by the corollary following Lemma 4.4 in \cite{ls78}: the $\GL_4(\mathbb{Z})$-orbits of cones in $\Sigma_4^{\perf}$ are equal to the $\SL_4(\mathbb{Z})$-orbits.
 Second, a cone which is $\SL_4(\mathbb{Z})$-alternating may no longer be $\GL_4(\mathbb{Z})$-alternating. We now check whether this occurs for the cones in Table~\ref{tab:p4cones}. 

In degree $9$, neither cone is alternating since transposition matrices stabilize these cones but reverse orientation, as is observed in \cite[p. 107]{ls78}.
 In degrees 4, 5, and 8 the graphic matroids giving rise to each of the cones in Table \ref{tab:p4cones} have an automorphism coming from an odd permutation of the ground set elements, so these cones are not alternating. Therefore $V^{(4)}_9 = V^{(4)}_8  = V^{(4)}_5 = V^{(4)}_4 = 0$. In degree $6$, the cone $\sigma(\includegraphics[height=0.14 in]{sl_graph_6.pdf})$ is alternating because any automorphism of $M(\includegraphics[height=0.14 in]{sl_graph_6.pdf} = K_4 \cup \{e\})$ fixes $e$ and $\sigma(K_4)$ is alternating, so $V^{(4)}_6 = \mathbb{Q}$.

We now explain the nonzero morphisms. 
Since $V^{(4)}_\bullet$ is 0 in degree less than 2,
the map $P_0^{(4)} \rightarrow P_{-1}^{(4)}$ is an isomorphism. 
We now compute the map $P_6^{(4)} \rightarrow P_{5}^{(4)}$. We have that $P_6^{(4)}$ is generated by one cone $\sigma(\includegraphics[height=0.14 in]{sl_graph_6.pdf} = K_4 \cup \{e\})$. Its faces are the cones obtained from $K_4 \cup \{e\}$ by deleting one edge. Only deleting the edge $e$ yields a graph which gives an alternating perfect cone, and this cone  generates $P_{5}^{(4)}$. So, this map is an isomorphism.
\end{proof}

\begin{theorem}\label{thm:top-weight-g4}
The top-weight cohomology  $\Gr_{20}^WH^i(\mathcal{A}_4;\mathbb{Q})$ of $\mathcal{A}_4$ is 0 for all $i$.
\end{theorem}

\proofnow{The top-weight cohomology of $\mathcal{A}_4$ is given by the homology of the chain complex $P_\bu^{(4)}$ by Theorem \ref{thm:correspondence}. This chain complex has no homology.}

In fact, our explicit description of $P^{(4)}_\bullet$ shows that $P^{(4)}_\bullet = I^{(4)}_\bullet$, since every nonzero generator is either of rank $<4$ or has a coloop.  The acyclicity of $P^{(4)}$ is then consistent with Theorem~\ref{thm:Ig-acyclic}.

\begin{remark}\label{rem:a4vanishes}
Theorem \ref{thm:top-weight-g4} can be deduced from the results in \cite{hulek-tommasi-cohomology}. In particular, the weight $0$ compactly supported cohomology of $\cA_4$ is encoded in the last two columns of Table 1 in loc. cit., which describes the first page of a spectral sequence converging to the cohomology of the second Voronoi compactification $\overline\cA_4^{\rm Vor}$ of $\cA_4$.
Here, these two columns contain the compactly supported cohomology of two strata of $\ov\cA_4^{\rm Vor}$ whose union is exactly $\cA_4$: the fifth column corresponds to the Torelli locus, while the sixth column corresponds to its complement in $\cA_4$.
By Poincar\'e duality \eqref{eq:pairing}, as described in Section \ref{S:relations}, if $\cA_4$ had top-weight cohomology it would also have compactly supported cohomology in weight $0$.  
However, even though there are some undetermined entries in the sixth column of the aforementioned table, 
 a close look at the table shows that the weight $0$ part must vanish. Indeed, there are no weight $0$ classes in the table northwest of the undetermined entries, so any weight $0$ classes in the sixth column would persist in the $E_\infty$ page of the spectral sequence and yield weight $0$ classes  of $\ov\cA_4^{\rm Vor}$. But this is impossible as $\ov\cA_4^{\rm Vor}$ is a smooth compactification of $\cA_4$. 
\end{remark}

\medskip

\subsection{The complex \texorpdfstring{$P_\bu^{(5)}$}{P5}}
By using the short exact sequence given in Theorem~\ref{thm:exact}, we now compute the complex $P^{(5)}_\bullet$. From this we compute the top-weight cohomology of $\mathcal{A}_{5}$.

\begin{proposition}
\label{prop:p5}
The chain complex $P_\bu^{(5)}$ is
\begin{center}
    \begin{tikzcd}[row sep = .5em, column sep = 1.5em]
    & P_{14}^{(5)} & P_{13}^{(5)} & P_{12}^{(5)} & P_{11}^{(5)} & P_{10}^{(5)} & P_9^{(5)} & P_8^{(5)} & P_7^{(5)} \\
    0 \rar & \QQ^3 \rar{\partial_{14}} & \QQ^2 \rar & 0 \rar & \QQ \rar{\partial_{11}} & \QQ^6 \rar{\partial_{10}} & \QQ^7 \rar{\partial_{9}} & \QQ \rar & 0 \rar &\text{ }\\
    & P_6^{(5)} & P_5^{(5)} &P_4^{(5)} &P_3^{(5)} &P_2^{(5)} &P_1^{(5)} &P_0^{(5)} &P_{-1}^{(5)} \\
    & \QQ \rar{\sim} & \QQ \rar & 0 \rar & 0 \rar & 0 \rar & 0 \rar & \QQ \rar{\sim} & \QQ \rar & 0. \\
    \end{tikzcd}
\end{center}
\end{proposition}
\begin{proof}
 By Theorem \ref{thm:exact}, we have in any degree $\ell$ that $\dim P_{\ell}^{(5)} =\dim P_{\ell}^{(4)} +\dim  V^{(5)}_{\ell}$. We have already computed $P^{(4)}_\bullet$, so we now study $V^{(5)}_{\bullet}$, which was computed in \cite{elbaz-vincent-gangl-soule-perfect}.
Recall from Section \ref{S:PerfectChainComplex} that $\Gamma_n$ denotes the set of representatives of alternating perfect cones of dimension $n+1$.
 In \cite[Table 1]{elbaz-vincent-gangl-soule-perfect},
 the cardinality of  $\Gamma_n$,
is given by 
$$
\begin{array}{|c|c|c|c|c|c|c|c|c|c|c|c|}\hline \mathbf{n} & 4 & 5 & 6 & 7 & 8 & 9 & 10 & 11 & 12 & 13 & 14  \\ \hline |\Gamma_n| & 0 & 0 & 0 & 0 & 1 & 7 & 6 & 1 & 0 & 2 & 3 \\ \hline\end{array}.
$$
In  \cite[Section 6.2]{elbaz-vincent-gangl-soule-perfect}, there is an explicit description of the differential maps.

Since $V_\bullet^{(5)}$ is supported in degrees $>7$, while $P_\bullet^{(4)}$ is supported in degrees $<7$, the differential maps $P^{(5)}_j \rightarrow P^{(5)}_{j-1}$ for $j<7$ are inherited from $P^{(4)}_\bullet$, and likewise the differential maps $P^{(5)}_j \rightarrow P^{(5)}_{j-1}$ for $j>7$ are inherited from $V^{(5)}_{\bullet}$. 
\end{proof}

\thmnow{The top-weight cohomology of $\mathcal{A}_5$ is 
$$
\Gr_{30}^WH^i(\mathcal{A}_5;\mathbb{Q})=
\begin{cases}
\mathbb{Q} \text{ if }i=15\text{ or }20,\\
0\text{ else.}
\end{cases}
$$
\label{thm:a5cohomology}
}

\begin{proof}
By Proposition \ref{prop:p5} and \cite[Theorem 4.3]{elbaz-vincent-gangl-soule-perfect} we have that $H_9(P^{(5)}_\bullet) = \mathbb{Q}$ and $H_{14}(P^{(5)}_\bullet) = \mathbb{Q}$.
Then by Theorem \ref{thm:correspondence}, we obtain the desired result. 
\end{proof}

\remnow{Grushevsky asks if $\mathcal{A}_g$ ever has nonzero odd cohomology \cite[Open Problem 7]{grushevsky-geometry}. Theorem \ref{thm:a5cohomology} confirms that $\mathcal{A}_5$ does in degree 15.
Furthermore, we will see in Theorem~\ref{thm:top-weight-g7} that Grushevsky's question is also answered affirmatively for $\cA_7$, where $$\dim \Gr^W_{56}H^{33}(\cA_7;\QQ ) = \dim \Gr^W_{56}H^{37}(\cA_7;\QQ ) = 1. $$

}

\medskip

\subsection{The top-weight cohomology of $\mathcal{A}_6$ and $\mathcal{A}_{7}$} 

In \cite[Theorem~4.3]{elbaz-vincent-gangl-soule-perfect}\footnote{Elbaz-Vincent, Gangl, and Soul\'{e} define the Voronoi complex as a complex of free $\ZZ$-modules, and in \cite[Theorem~4.3]{elbaz-vincent-gangl-soule-perfect} they compute the integral homology of this complex. Our definition of the Voronoi complex $V^{(g)}_{\bullet}$ is a complex of $\QQ$-vector spaces, but this causes no problems as we are only interested in the rational homology of $V^{(g)}_{\bullet}$.} Elbaz-Vincent, Gangl, and Soul\'{e} computed the homology of the Voronoi complex $V^{(g)}_{\bullet}$ for $g=5,6,$ and $7$. Combining this, together with Proposition~\ref{prop:p5}, we are able to compute the top-weight cohomology of $\mathcal{A}_6$ and $\mathcal{A}_{7}$.

\begin{theorem}\label{thm:top-weight-g6}
The top-weight cohomology of $\mathcal{A}_{6}$ is 
\[
\Gr_{42}^WH^i(\mathcal{A}_6;\mathbb{Q})=
\begin{cases}
\mathbb{Q} \text{ if }i=30,\\
0\text{ else.}
\end{cases}
\]
\end{theorem}

\begin{proof}
By Proposition~\ref{prop:pg-cellular-chain}, we need to show that $H_{11}(P^{(6)}_{\bullet})\cong\QQ$ and $H_{i}(P^{(6)}_{\bullet})=0$ for $i\neq11$. Consider the long exact sequence in homology arising from the short exact sequence of chain complexes given in Theorem~\ref{thm:exact}. Combining this with the computation of the homology of $V^{(6)}_{\bullet}$ \cite[Theorem 4.3]{elbaz-vincent-gangl-soule-perfect} and the homology of $P^{(5)}_{\bullet}$ given in Proposition~\ref{prop:p5}, our computation of $H_k(P^{(6)}_{\bullet})$ reduces to the four cases in Table \ref{tab:a6}.

\begin{table}[h]
\begin{center}
\begin{tabular}{ | c | c | c | c | }
\hline
$i$ & $H_{i}(P^{(5)}_{\bullet} )$ & $H_{i}(P^{(6)}_{\bullet} )$ & $H_{i}(V^{(6)}_{\bullet} )$ \\ \hline

$\geq16$     &  0         &   {\color{BurntOrange} 0}         &  0           \\ \hline 
15                &  0         &   {\color{Blue} 0}      &  $\QQ$     \\ \hline 
14                &   $\QQ$  &   {\color{Blue} 0}      &  0           \\ \hline 
13                &   0        &   {\color{BurntOrange} 0}          &  0           \\ \hline 
12                &   0        &   {\color{BurntOrange} 0}          &  0           \\ \hline 
11                &   0        &   {\color{Red} $\QQ$}                   &  $\QQ$     \\ \hline 
10                &   0        &   {\color{Green} 0}    &  $\QQ$      \\ \hline 
9                  &   $\QQ$  &    {\color{Green} 0}   &  0            \\ \hline 
$\leq 8$       &  0         &    {\color{BurntOrange} 0}          &  0            \\ \hline
\end{tabular}
\caption{The long exact sequence in homology for $g=6$.}
\label{tab:a6}
\end{center}
\end{table}
\begin{itemize}
\item {\color{BurntOrange} \textbf{Case 1:} $(i\leq8, i=12,13, i\geq16)$:} For these values of $i$, both $H_{i}(P^{(5)}_{\bullet} )$ and $H_{i}(V^{(6)}_{\bullet} )$  are equal to zero, 
so 
$H_{i}(P^{(6)}_{\bullet} )=0.$

\item {\color{Blue} \textbf{Case 2:} $(i=14,15)$:} The long exact sequence in homology 
gives the exact sequence
\begin{center}
    \begin{tikzcd}
0 \rar{} & H_{15}(P^{(6)}_{\bullet} ) \rar{} & \QQ \rar{\delta^{6}_{15}} & \QQ \rar{} & H_{14}(P^{(6)}_{\bullet} ) \rar{} & 0. 
\end{tikzcd}
\end{center}
Exactness implies that the connecting homomorphism $\delta^{6}_{15}$ is either an isomorphism or the zero map. By \cite[Theorem~6.1]{elbaz-vincent-gangl-soule-perfect} we know that inflating the cones in $V^{(5)}_{\bullet}$ gives an isomorphism of chain complexes $V^{(6)}_{\bullet} \cong V^{(5)}_{\bullet}[1]\oplus F_{\bullet}$ for some complex $F_{\bullet}$. Combining this with \cite[Theorem~4.3]{elbaz-vincent-gangl-soule-perfect} shows the nontrivial homology class in $H_{15}(V^{(6)}_{\bullet})$ is the inflation of a nontrivial homology class in $H_{14}(V^{(5)}_{\bullet})$. By Proposition~\ref{prop:p5}, the nontrivial homology class in $H_{14}(P^{(5)}_{\bullet})$ is the nontrivial homology class in $H_{14}(V^{(5)}_{\bullet})$, so $H_{15}(V^{(6)}_{\bullet})$ is generated by the inflation of the nontrivial class $H_{14}(P^{(5)}_{\bullet})$. By the proof of the acyclicity of the inflation complex $I^{(g)}_{\bullet}$ (Theorem~\ref{thm:Ig-acyclic}), this implies the connecting map $\delta_{15}^{6}$ is an isomorphism. The exact sequence above then implies that $H_{k}(P^{(6)}_{\bullet} )=0$ for both $k=14$ and $k=15$.

\item {\color{Red} \textbf{Case 3:} $(i=11)$:} Since $H_{11}(P^{(5)}_{\bullet} )$ and $H_{10}(P^{(5)}_{\bullet} )$ vanish, the long exact sequence in homology gives us 
\begin{center}
\begin{tikzcd}
0 \rar{} & H_{11}(P^{(6)}_{\bullet} ) \rar{} & \QQ \rar{} & 0. 
\end{tikzcd}
\end{center}
This exactness implies that $H_{11}(P^{(6)}_{\bullet} )$ is isomorphic to $\QQ$. 

\item  {\color{Green} \textbf{Case 4:} $(i=9,10)$:} By considering the long exact sequence in homology in the range $i=10$ to $i=9$ we have the following exact sequence:
\begin{center}
\begin{tikzcd}
0 \rar{} & H_{10}(P^{(6)}_{\bullet} ) \rar{} & \QQ \rar{\delta^{6}_{10}} & \QQ \rar{} & H_{9}(P^{(6)}_{\bullet} ) \rar{} & 0. 
\end{tikzcd}
\end{center}
An analysis similar to that in Case 2 shows that connecting map $\delta_{11}^{6}$ is an isomorphism, implying by exactness that $H_{i}(P^{(6)}_{\bullet} )=0$ for both $i=9$ and $i=10$.

\end{itemize}
\end{proof}

We now compute the top-weight rational cohomology of $\mathcal{A}_7$.
 
\begin{theorem}\label{thm:top-weight-g7}
The top-weight cohomology of $\mathcal{A}_{7}$ is 
\[
\Gr_{56}^WH^i(\mathcal{A}_7;\mathbb{Q})=
\begin{cases}
\mathbb{Q} \text{ if }i=28,33,37,42\\
0\text{ else.}
\end{cases}
\]
\end{theorem}

\begin{proof}
We compute the homology of $P^{(7)}_\bullet$ in a similar fashion to the proof of Theorem~\ref{thm:top-weight-g6}, by considering the long exact sequence in homology arising from the short exact sequence of chain complexes given in Theorem~\ref{thm:exact}. Table \ref{tab:a7} records the homology of $P^{(6)}_{\bullet}$ and $V^{(7)}_{\bullet}$, which are given in Table \ref{tab:a6} and \cite[Theorem 4.3]{elbaz-vincent-gangl-soule-perfect} respectively.

\begin{table}[h]
\begin{center}
\begin{tabular}{ | c | c | c | c | }
\hline
$i$ & $H_{i}(P^{(6)}_{\bullet} )$ & $H_{i}(P^{(7)}_{\bullet} )$ & $H_{i}(V^{(7)}_{\bullet} )$ \\ \hline

$\geq28$     &  0                                            &    {\color{BurntOrange} 0}  			&  0            \\ \hline
27                &  0                                            &    {\color{red} $\QQ$}  				&  $\QQ$	 \\ \hline
26                &  0                                            &    {\color{BurntOrange} 0}  			&  0            \\ \hline 
25                &  0                                            &    {\color{BurntOrange} 0}  			&  0            \\ \hline 
24                &  0                                            &    {\color{BurntOrange} 0}  			&  0            \\ \hline 
23                &  0                                            &    {\color{BurntOrange} 0}  			&  0            \\ \hline 
22                &  0                                            &    {\color{red} $\QQ$}   				&  $\QQ$	  \\ \hline
21                &  0                                            &    {\color{BurntOrange} 0}  			& 0            \\ \hline 
20                &  0                                            &    {\color{BurntOrange} 0}  			& 0            \\ \hline
19                &  0                                            &    {\color{BurntOrange} 0}  			& 0            \\ \hline
18                &  0                                            &	 {\color{red} $\QQ$}   				&  $\QQ$	 \\ \hline
17                &  0                                            &    {\color{BurntOrange} 0}  			&   0          \\ \hline
16                &  0                                            &   {\color{BurntOrange} 0}         		&  0           \\ \hline 
15                &  0      &   {\color{BurntOrange} 0}     		&  0          \\ \hline 
14                &  0     &   {\color{BurntOrange} 0}      		&  0           \\ \hline 
13                &   0                                            &   {\color{red} $\QQ$}          			&  $\QQ$     \\ \hline 
12                &   0                                            &    {\color{Green}$0$}             &  $\QQ$     \\ \hline 
11                &   $\QQ$                                       &    {\color{Green}$0$}            &  0           \\ \hline 
10                &   0   &   {\color{BurntOrange} 0}     &  0           \\ \hline 
9                  &   0   &    {\color{BurntOrange} 0}    &  0            \\ \hline 
$\leq 8$       &  0                                              &    {\color{BurntOrange} 0}                   &  0            \\ \hline
\end{tabular}
\caption{The long exact sequence in homology for $g=7$.}
\label{tab:a7}
\end{center}
\end{table}
Both {\color{BurntOrange} \textbf{Case 1} $(i\neq 11,12,13,18,22,27)$} and {\color{Red} \textbf{Case 2} $(i=13,18,22,27)$} follow from the exactness of the long exact sequence on homology in a manner analogous to Cases 1 and 3 in the proof of Theorem~\ref{thm:top-weight-g6}.

For {\color{Green} \textbf{Case 3} $(i=11,12)$}, the long exact sequence in homology 
 gives the exact sequence
\begin{center}
\begin{tikzcd}
0 \rar{} & H_{12}(P^{(7)}_{\bullet}) \rar{} & \QQ \rar{\delta_{12}^{7}} & \QQ \rar &  H_{11}(P^{(7)}_{\bullet}) \rar & 0. 
\end{tikzcd}
\end{center}
Now $\delta_{12}^{7}$ is either an isomorphism or it is the zero map. As discussed in \cite[\S6.3]{elbaz-vincent-gangl-soule-perfect}, the nontrivial homology class in $H_{12}(V^{(7)}_{\bullet})$ is the inflation of a nontrivial homology class in $H_{11}(V^{(6)}_{\bullet})$. However, since by the proof of Theorem~\ref{thm:top-weight-g6}, the nontrivial homology class in $H_{11}(P^{(6)}_{\bullet})$ is the nontrivial homology class in $H_{11}(V^{(6)}_{\bullet})$, this implies that $H_{12}(V^{(7)}_{\bullet})$ is generated by the inflation of the nontrivial class $H_{11}(P^{(6)}_{\bullet})$. By the proof of the acyclicity of the inflation complex $I^{(g)}$ (Theorem~\ref{thm:top-weight-g6}), this implies the connecting map $\delta_{12}^{7}$ is an isomorphism. The exact sequence above then implies $H_{k}(P^{(7)}_{\bullet} )=0$ for  $i=11$ and $i=12$.

\end{proof}

 \noindent Theorem \ref{thm:main} now follows directly from Theorems~ \ref{thm:top-weight-g4}, \ref{thm:a5cohomology}, \ref{thm:top-weight-g6}, and  \ref{thm:top-weight-g7}. As a corollary of this we are able to deduce the top-weight Euler characteristic of $\cA_{g}$ for $2\leq g \leq 7$. 
 
 \begin{corollary}
 The top-weight Euler characteristic of $\mathcal{A}_g$ for $2\leq g \leq 7$ is
 \[
 \chi^{\rm{top}}(\mathcal{A}_g)=\begin{cases} 1 &\mbox{if } g=3,6 \\ 
0 & \mbox{if }g=2,4,5,7. \end{cases}
 \]
 \end{corollary}

\begin{remark}
One can also deduce the top-weight Euler characteristic of $\mathcal{A}_{g}$ for $5 \leq g \leq 7$ directly from the numbers listed in \cite[Figures 1 and 2]{elbaz-vincent-gangl-soule-perfect}. 
It would be interesting to know whether a closed formula for the top-weight Euler characteristic of $\mathcal{A}_{g}$ exists in general.
\end{remark}

\begin{remark}\label{rem:non-vanishing}
We have established 
\begin{equation}\label{conj:bold-conjecture}
\Gr^W_{(g+1)g} H^{g(g-1)}(\cA_g;\QQ)\ne 0
\end{equation}
for $g=3, 5,6,$ and $7$ ($g=3$ also follows from  \cite{hain-rational}).  We ask whether~\eqref{conj:bold-conjecture} holds for all $g\ge 5$.  Equivalently, the question is whether $H_{2g-1}(P^{(g)}) \ne 0$ for all $g\ge 5$.  
The connection to the stable cohomology of the Satake compactification, as summarized in Table~\ref{tab:E1}, gives evidence for this question, as explained in Section \ref{S:relations}; see Question~\ref{q:questions}.
We also note the possible relationship with the main theorems of \cite{cgp-graph-homology} on the rational cohomology of $\cM_g$, which use the fact that $H_{2g-1}(G^{(g)})\ne 0$ for $g=3$ and $g\ge 5$ (\cite{brown-mixed}, \cite{willwacher-kontsevich}, see \cite[Theorem 2.7]{cgp-graph-homology}). 
We leave this interesting investigation as an open question.
\end{remark}

\medskip

\subsection{Results for $g \geq 8$}\label{subsec:g8}

While full calculations for the top-weight cohomology of $\cA_g$ in the range $g\ge 8$ are beyond the scope of current computations, we can nevertheless use our previous computation of the top-weight cohomology of $\mathcal{A}_{7}$ together with a vanishing result of \cite{sikiri2019voronoi} to show that the top-weight cohomology of $\mathcal{A}_{8}, \mathcal{A}_{9},$ and $\mathcal{A}_{10}$ vanishes in a certain range slightly larger than what is given by the virtual cohomological dimension.

\begin{theorem}\label{thm:top-weight-vanishing-g8910}
The top-weight rational cohomology of $\mathcal{A}_{8}$, $\mathcal{A}_{9}$, and $\mathcal{A}_{10}$ vanishes in the following ranges:
\begin{align*}
    \Gr_{72}^WH^i(\mathcal{A}_{8};\mathbb{Q})&=0 \quad \text{for $i\geq 60$} \\
    \Gr_{90}^WH^i(\mathcal{A}_{9};\mathbb{Q})&=0 \quad \text{for $i\geq 79$} \\
    \Gr_{110}^WH^i(\mathcal{A}_{10};\mathbb{Q})&=0 \quad \text{for $i\geq 99$}.
\end{align*}
\end{theorem}

\begin{proof}
By Theorem~4.5 of \cite{sikiri2019voronoi} for $g=8,9,$ and $10$ the homology $H_{i}(V^{(g)}_{\bullet})=0$ for $i\leq 11$, and further, $H_{12}(V^{(8)}_{\bullet})=0$. Considering the long exact sequence in homology:
\[
\begin{tikzcd}
\cdots \rar{} & H_{i+1}\left(V^{(g)}_{\bullet}\right) \rar{\delta} & H_{i}\left(P^{(g-1)}_{\bullet}\right) \rar{} & H_{i}\left(P^{(g)}_{\bullet}\right) \rar{} & H_{i}\left(V^{(g)}_{\bullet}\right) \rar{} & \cdots
\end{tikzcd}
\]
coming from the short exact sequence of chain complexes given in Theorem~\ref{thm:exact}, we see that this vanishing implies that  
\[
H_{i}(P^{(7)}_{\bullet})\cong H_{i}(P^{(8)}_{\bullet}) \cong H_{i}(P^{(9)}_{\bullet}) \cong H_{i}(P^{(10)}_{\bullet})
\]
for $i\leq 10$ and $H_{11}(P^{(7)}_{\bullet})\cong H_{11}(P^{(8)}_{\bullet})$. By our computation of the homology of $P^{(7)}_{\bullet}$ in the proof of Theorem~\ref{thm:top-weight-g7} we know that $H_{i}(P^{(7)}_{\bullet})=0$ for all $i\leq 12$, implying that for $g=8,9,$ and $10$, the homology $H_{i}(P^{(g)}_{\bullet})=0$ for $i\leq 10$, and further, $H_{11}(P^{(8)}_{\bullet})=0$. The result now follows from Proposition~\ref{prop:pg-cellular-chain}. 
\end{proof}

\begin{remark}
These vanishing bounds for $g=8,9,10$ are slightly larger than the bounds provided by Corollary~\ref{cor:low-degree-vanishing}, equivalently, the fact that $\on{vcd} \cA_g = g^2$ (see Remark~\ref{rem:vcd}), which imply that 
\begin{align*}
    \Gr_{72}^WH^{i} (\mathcal{A}_{8};\mathbb{Q})&=0 \quad \text{for $i\geq 65$} \\
    \Gr_{90}^WH^{i} (\mathcal{A}_{9};\mathbb{Q})&=0 \quad \text{for $i\geq 82$} \\
    \Gr_{110}^WH^{i}(\mathcal{A}_{10};\mathbb{Q})&=0 \quad \text{for $i\geq 101$}.
\end{align*}
The result for $g=10$, however, is subsumed by the more general fact that the top-weight cohomology of $\cA_g$ vanishes in degrees $0$ and $1$ below the vcd, as we shall note in \S7 below.
\end{remark}

\section{Relationship with the stable cohomology of $\Satake$.}\label{S:relations}

Our results on the existence of certain top-weight cohomology classes of $\cA_g$ can be related to results of Chen-Looijenga \cite{chen-looijenga-stable} and Charney-Lee \cite{charney-lee-cohomology}, which predict that, as $g$ grows, there should be infinitely many of these classes.
This connection was brought to our attention by O.~Tommasi, and we thank her for explaining 
her ideas to us in detail.

Recall that $\Ag$ admits a compactification $\Satake$, called the Satake or Baily-Borel compactification,
first constructed as a projective variety  by Baily and Borel in \cite{baily-borel-compactification}.
This compactification can be seen as a minimal compactification in the sense that it admits a morphism from all toroidal compactifications of $\cA_g$.
The reader interested in learning more about the vast literature on $\cA_g$ and its compactifications can look at the very nice surveys \cite{grushevsky-geometry} and \cite{hulek-tommasi-topology}.
There are natural maps $\mathcal{A}_g^{\mathrm{Sat}}\to \mathcal{A}_{g+1}^{\mathrm{Sat}}$, and the groups $H^k(\cA_g^{\mathrm{Sat}};\QQ)$ stabilize for $k<g$ \cite{charney-lee-cohomology}.
Moreover, as Charney-Lee prove, the stable cohomology ring $H^{\bullet}(\cA^{\mathrm{Sat}}_\infty;\QQ)$ of the Satake compactifications is freely generated by the classes $\lambda_i$ for $i$ odd, and the classes $y_{4j+2}$ for $j=1,2,3,\ldots$, where $y_{4j+2}$ is in degree $4j+2$.  Here, the $\lambda$-classes extend the $i^\mathrm{th}$ Chern class of the Hodge bundle on $\cA_g$; in particular they are algebraic, and hence never have weight $0$. But the classes $y_j$ have weight $0$, as proven recently by Chen-Looijenga \cite{chen-looijenga-stable}. This result is very important in the discussion that follows. 

Recall also that $\Satake$ admits a stratification by locally closed substacks
$$\Satake = \cA_g \sqcup \cA_{g-1} \sqcup \cdots \sqcup \mathcal{A}_0.$$
Thus the spectral sequence on compactly supported cohomology associated to this stratification is
\begin{equation}\label{eq:ss}
E_1^{p,q} = H^{p+q}_c(\cA_p;\QQ)  \Rightarrow H^{p+q}(\Satake;\QQ),
\end{equation}
where $p=0,\ldots,g$.  
This spectral sequence may be interpreted in the category of mixed Hodge structures. Passing to the weight $0$ subspace, we see that the existence of the products of the $y_j$ classes in the stable cohomology ring of the Satake compactification implies the existence of infinitely many cohomology classes  in $\Gr^W_0 \!H^{j}_c (\cA_g;\QQ)$ for all $g$, and hence by the perfect pairing 
\begin{equation}\label{eq:pairing} \Gr^W_0 \!H^{j}_c (\cA_g;\QQ) \times \Gr^W_{(g+1)g} H^{(g+1)g-j}(\cA_g;\QQ) \to \QQ\end{equation}
provided by Poincar\'e duality, infinitely many classes $\Gr^W_{(g+1)g} H^*(\cA_g;\QQ)$ in top weight.  

With Poincar\'e duality applied, all of the known results on the top-weight cohomology of $\cA_g$, including our Theorems~\ref{thm:a5cohomology},  \ref{thm:top-weight-g6}, and \ref{thm:top-weight-g7},
can thus be summarized in Table~\ref{tab:E1}, which shows the weight 0 part of the $E_1$ page of the spectral sequence~\eqref{eq:ss}.

\begin{table}[h]\label{table:spectral-sequence}
    \centering
    \begin{tabular}{c|cccccccccccc}
 21   & 0 & 0 & 0 & 0 & 0 & 0 & 0 & $\QQ$ &  &  &  & \\ 
 20   & 0 & 0 & 0 & 0 & 0 & 0 & 0 & 0 &  &  & &  \\ 
 19   & 0 & 0 & 0 & 0 & 0 & 0 & 0 & 0 &  &  &  &\\ 
 18   & 0 & 0 & 0 & 0 & 0 & 0 & 0 & 0 &  &  & & \\ 
 17   & 0 & 0 & 0 & 0 & 0 & 0 & 0 & 0 &  &  & & \\ 
 16   & 0 & 0 & 0 & 0 & 0 & 0 & 0 & $\QQ$ &  &  & & \\ 
 15   & 0 & 0 & 0 & 0 & 0 & 0 & 0 & 0 &  &  &  &\\ 
 14   & 0 & 0 & 0 & 0 & 0 & 0 & 0 & 0 &  &  & &\\ 
 13   & 0 & 0 & 0 & 0 & 0 & 0 & 0 & 0 &  &  &  &\\ 
 12   & 0 & 0 & 0 & 0 & 0 & 0 & 0 & $\QQ$ &  &  & & \\ 
 11   & 0 & 0 & 0 & 0 & 0 & 0 & 0 & 0 &  &  &  &\\ 
 10   & 0 & 0 & 0 & 0 & 0 & $\QQ$ & 0 & 0 &  &  & & \\ 
 9   & 0 & 0 & 0 & 0 & 0 & 0 & 0 & 0 &  &  &  &\\ 
 8   & 0 & 0 & 0 & 0 & 0 & 0 & 0 & 0 &  &  &  &\\ 
 7   & 0 & 0 & 0 & 0 & 0 & 0 & 0 & $\QQ$ &  &  &  &\\ 
 6   & 0 & 0 & 0 & 0 & 0 & 0 & $\QQ$ & 0 &  &  &  &\\ 
 5   & 0 & 0 & 0 & 0 & 0 & $\QQ$ & 0 & 0 &  &  &  &\\ 
 4   & 0 & 0 & 0 & 0 & 0 & 0 & 0 & 0 & 0 &  &  &\\ 
 3   & 0 & 0 & 0 & $\QQ$ & 0 & 0 & 0 & 0 & 0 &  &  &\\ 
 2   & 0 & 0 & 0 & 0 & 0 & 0 & 0 & 0 & 0 & 0 &  &\\ 
 1   & 0 & 0 & 0 & 0 & 0 & 0 & 0 & 0 & 0 & 0 & 0 &$\cdots$\\ 
 0   & $\QQ$ & 0 & 0 & 0 & 0 & 0 & 0 & 0 & 0 & 0 & 0& $\cdots$\\ 
    \hline
    & 0 & 1 & 2 & 3 & 4 & 5 & 6 & 7 & 8 & 9 & 10 &$\cdots$\\
    \end{tabular}
    \medskip
    \caption{The page $E_1^{p,q} = \Gr^W_0\! H^{p+q}_c(\cA_p;\QQ)\Rightarrow \Gr^W_0\! H^{p+q}(\Satake;\QQ)$ of the Gysin spectral sequence, for $g$ sufficiently large.  The blank entries for $p\ge 8$  are currently unknown.}
    \label{tab:E1}
\end{table}

Implicit in Table~\ref{tab:E1} is the fact that all terms below the $p$-axis are zero. This follows from the fact that $\on{vcd}(\cA_g) = \on{vcd}(\on{Sp}(2g,\ZZ)) = g^2$, or, just as well, from the fact that $\on{vcd}(\GL_g(\ZZ)) = {g\choose 2}$ \cite{borel-serre-corners}.  In fact, the vanishing below the $p$-axis as well as in the rows $q=0$, $1$, and $2$, apart from $(p,q)=(0,0)$, can be deduced from the fact that the rational cohomology of $\GL_g(\ZZ)$ vanishes in degrees $0$, $1$, and $2$ below the vcd. Indeed, we have, for all $k$,
\[H^{{g\choose 2}-k}(\GL_g(\ZZ);\QQ) \cong H_k(\GL_g(\ZZ);\mathrm{St}\otimes \QQ) \cong H_{k+g-1}(V^{(g)})\]
where $\mathrm{St}$ denotes the Steinberg module \cite{soule-3-torsion}; these
are all zero when $g>1$ for $k=0$ \cite{lee-szczarba-homology}, $k=1$ \cite{church-putman-codimension}, 
and $k=2$ \cite{bruck-miller-patzt-sroka-wilson-codimension}.  Then Theorem~\ref{thm:exact} implies that also
$H_{k+g-1}(P^{(g)})=0$ and $k = 0, 1,$ and $2$ so also $\Gr^W_{g^2+g} H^{g^2-k}(\cA_g;\QQ) = (\Gr^W_0 \!H_c^{g+k}(\cA_g;\QQ))^\vee= 
0 $ for $g>0$ and $k\le 2$
by Proposition \ref{prop:pg-cellular-chain}.  

 As explained to us by Tommasi, the classes in Theorems~\ref{thm:a5cohomology},  \ref{thm:top-weight-g6}, and \ref{thm:top-weight-g7}, as well as the already-known class in $\Gr^W_{12} H^{6}(\cA_3;\QQ)$ \cite{hain-rational} give natural candidates for classes in $\Gr^W_0\! H^{p+q}_c(\cA_p;\QQ)$ that produce the classes $y_{4j+2}$ in the spectral sequence~\eqref{eq:ss}, in the sense that they persist in the $E_\infty$ page in the Gysin spectral sequence for $g$ sufficiently large.  Indeed, looking at the $p=q$ diagonal on the $E_1$ page of the spectral sequence in Table~\ref{tab:E1}, we are led to ask:

\begin{question}\label{q:questions}\mbox{}
\begin{enumerate}
    \item Is $\Gr^W_0\!H^{2g}_c(\cA_g;\QQ) \ne 0$ for $g=3$ and all $g\ge5$? 
    \item Moreover, do these cohomology classes produce the stable cohomology classes in $\Gr^W_0\!H^{\bullet}(\cA_\infty^{\mathrm{Sat}};\QQ)$?
    \item Is $\Gr^W_0\!H^{k}_c(\cA_g;\QQ) = 0$ for $k<2g$?
\end{enumerate}
\end{question}

As discussed in the introduction, an affirmative answer to the third question in the range $k<2g-1$ would be implied by \cite[Conjecture 2]{church-farb-putman-stability}.
Our Theorems \ref{thm:a3}, \ref{thm:a5cohomology},  \ref{thm:top-weight-g6} and \ref{thm:top-weight-g7} verify the first and third questions for $g\le 7$. They also verify the second question for $g=3$ and for $g=5$. Indeed, $\Gr^W_0\!H^6_c(\cA_3;\QQ)$  and $\Gr^W_0\!H^{10}_c(\cA_5;\QQ)$ are the only nonzero terms in the antidiagonals $p+q=6$ and $p+q=10$, respectively; so they produce the classes $y_6\in \Gr^W_0\!H^6(\cA_\infty^{\mathrm{Sat}};\QQ)$ and $y_{10}\in \Gr^W_0\! H^{10}(\cA_\infty^{\mathrm{Sat}};\QQ)$, respectively.  It is natural to guess that the other terms in Table~\ref{tab:E1} similarly produce products of $y_j$'s: for example, that $\Gr^W_0\!H^{12}_c(\cA_6;\QQ)$ produces $y_6^2$, and that $\Gr^W_0\!H^{14}_c(\cA_7;\QQ)$ produces $y_{14}$, and so on.

Finally, Tommasi also remarks that the odd degree classes in weight $0$ compactly supported cohomology of $\cA_g$ detected so far, namely 
$$\Gr^W_0\!H^{15}_c(\cA_5;\QQ), \,\Gr^W_0\!H^{19}_c(\cA_7;\QQ), \text{ and }\Gr^W_0\!H^{23}_c(\cA_7;\QQ),$$
must of course be killed by a differential on some page of the spectral sequence, since $\Satake$ has no weight $0$ stable cohomology in odd degrees. This implies the existence of some even degree classes in $\Gr^W_0\! H^{\bullet}_c(\cA_g;\QQ)$ which kill the odd degree classes and which are not related by this spectral sequence to the products of $y_j$s.  It would be very interesting to explicitly identify such classes.

\bibliographystyle{alpha}
\bibliography{my}

\end{document}